\documentclass[]{scrartcl}

\usepackage[utf8]{inputenc}
\usepackage[T1]{fontenc}
\usepackage{lmodern}
\usepackage[english]{babel}
\usepackage{amsmath,amsthm,amssymb}

\usepackage{thmtools}
\usepackage{mathtools}
\usepackage[autostyle]{csquotes}

\usepackage{mathrsfs}

\usepackage{nicematrix}

\usepackage{listings}

\usepackage{tikz}
\usetikzlibrary{cd}
\usetikzlibrary{tikzmark,quotes}

\usepackage{stmaryrd}

\usepackage[backend=biber,
style=alphabetic,
giveninits=true,
maxbibnames=4,
]{biblatex}
\usepackage[hypertexnames=false]{hyperref}
\usepackage[nameinlink]{cleveref}

\usepackage{xurl}

\usepackage{xcolor}
\hypersetup{
	colorlinks,
	linkcolor={red!50!black},
	citecolor={blue!50!black},
	urlcolor={blue!80!black},
	breaklinks=true
}

\theoremstyle{plain}
\newtheorem{thm}{Theorem}[section]
\newtheorem{la}[thm]{Lemma}
\newtheorem{prop}[thm]{Proposition}
\newtheorem{cor}[thm]{Corollary}

\newtheorem*{thm*}{Theorem}
\newtheorem*{la*}{Lemma}
\newtheorem*{prop*}{Proposition}
\newtheorem*{cor*}{Corollary}

\theoremstyle{definition}
\newtheorem{mydef}[thm]{Definition}
\newtheorem{ex}[thm]{Example}
\newtheorem{rk}[thm]{Remark}
\newtheorem*{rk*}{Remark}

\newtheorem*{exer*}{Exercise}

\theoremstyle{remark}

\newcommand{\C}{\mathbb{C}}
\newcommand{\Z}{\mathbb{Z}}
\newcommand{\N}{\mathbb{N}}
\newcommand{\Q}{\mathbb{Q}}

\newcommand{\mono}{\hookrightarrow}
\newcommand{\ot}{\leftarrow}
\DeclareMathOperator{\id}{id}

\DeclareMathOperator{\spec}{Spec}
\newcommand{\Spec}{\spec}

\DeclareMathOperator{\qcoh}{QCoh}

\DeclareMathOperator{\supp}{supp}
\DeclareMathOperator{\Hom}{Hom}

\DeclareMathOperator{\Sym}{Sym}

\DeclareMathOperator{\End}{End}
\DeclareMathOperator{\Ext}{Ext}
\DeclareMathOperator{\Rep}{Rep}

\def\SheafHom{\mathop{\mathcal{H}\!\mathit{om}}\nolimits}

\newcommand{\GL}{\mathrm{GL}}

\newcommand{\PGL}{\mathrm{PGL}}

\newcommand{\Gr}{\mathrm{Gr}}

\newcommand{\twoByTwoMatrix}[4]{\left(\begin{matrix}#1 & #2 \\ #3 & #4 \end{matrix}\right)}

\DeclareMathOperator{\Proj}{Proj}

\newcommand{\shom}{\SheafHom}

\newcommand{\dnull}{{\delta_0}}

\newcommand{\cupprod}{\smile}

\newcommand{\tsst}{{\theta\text{-}\mathrm{sst}}}

\newcommand{\pt}{\{\mathrm{pt}\}}

\DeclareMathOperator{\Coha}{CoHA}
\DeclareMathOperator{\Chowha}{ChowHA}
\DeclareMathOperator{\IC}{IC}
\DeclareMathOperator{\reldim}{reldim}
\DeclareMathOperator{\codim}{codim}
\DeclareMathOperator{\rank}{rk}
\DeclareMathOperator{\Coh}{Coh}
\DeclareMathOperator{\Tor}{Tor}
\DeclareMathOperator{\Sch}{Sch}

\DeclareMathOperator{\Vect}{Vect}
\DeclareMathOperator{\MHS}{MHS}
\DeclareMathOperator{\MHM}{MHM}
\DeclareMathOperator{\Reg}{Reg}
\DeclareMathOperator{\StProj}{StProj}

\newcommand{\Pbb}{\mathbb{P}}
\newcommand{\sst}{\mathrm{sst}}

\newcommand{\simp}{\mathrm{simp}}
\newcommand{\Mf}{\mathfrak{M}}
\newcommand{\dimvectd}{{\mathbf{d}}}
\newcommand{\dimvecte}{{\mathbf{e}}}
\newcommand{\mycomment}[1]{}

\newcommand{\gitquotient}{/\!\!/}

\newcommand{\Pp}{\mathscr{P}}

\SetSymbolFont{stmry}{bold}{U}{stmry}{m}{n}

\numberwithin{equation}{section}

\allowdisplaybreaks

\title{CoHAs of Torsion Sheaves on Weighted Projective Curves}
\author{Timm Peerenboom}

\begin{document}
	
		\maketitle

\begin{abstract}
	\noindent \textbf{Abstract.} We describe the cohomological Hall algebra of torsion sheaves on a weighted projective line with weights~$(2, \dots, 2)$ in terms of generators and relations.
\end{abstract}

\section*{Introduction} \addcontentsline{toc}{section}{Introduction}

The cohomological Hall algebra (CoHA) was introduced by Kontsevich and Soibelman in \cite{KSCoHADef} as a ``categorification'' of Donaldson--Thomas invariants for quivers with stability and potential.
Already the case of zero potential proved to be difficult to compute.
For a symmetric quiver with zero potential and trivial stability, it has been shown by Efimov in \cite{EfimovTheorem} that the CoHA is a free symmetric algebra on a (super) vector space and thus encodes the DT-invariants.
Now, one can add a stability condition and obtain a local version of the full CoHA.
If the Euler form restricted to dimension vectors of a fixed slope is symmetric -- a generalization of the quiver being symmetric to the semistable case -- the Poincaré series of these algebras is still the same as the Poincaré series of a free symmetric algebra, however it has been shown in \cite[Section~10.2]{FRChowHa} that for the Kronecker quiver~$K_2 =\begin{tikzcd}
	\bullet \rar[bend left] \rar[bend right] & \bullet
\end{tikzcd}$ and the ``standard'' stability condition (which guarantees that semistable representations are regular representations) the CoHA is not isomorphic to a free symmetric algebra.
With the same methods, Franzen and Reineke in \cite{FRKronecker} gave a description of the semistable CoHA of~$K_2$ explicitly in terms of generators and relations.

A classical theorem of Beilinson \cite{BeilinsonDbPn} states that we have an equivalence of triangulated categories $D^b(\Pbb^1) \simeq D^b(K_2)$; this is induced by a tilting object in $D^b(\Pbb^1)$ whose endomorphism ring is the path algebra of $K_2$.
Restricting this equivalence to semistable/regular representations of $K_2$, we obtain an equivalence of abelian categories between~$\Reg(K_2)$ and~$\Tor(\Pbb^1)$, the category of torsion sheaves on~$\Pbb^1$. 
This equivalence of categories induces an isomorphism between the associated CoHAs; here the CoHA of the stack of torsion sheaves of~$\Pbb^1$ has been computed by Schiffmann and Vasserort in \cite{SchiffmannCohaOfCurves} independently of \cite{FRKronecker}.

There is a generalization of this equivalence to coherent sheaves of weighted projective lines~$\Pbb^1(\boldsymbol{\lambda}; \mathbf{w})$ for~$\boldsymbol{\lambda} = (\lambda_1, \dots, \lambda_n) \in \Pbb^1(\C)^n$ and weights~$\mathbf{w} = (w_1,\dots, w_n)\in \Z_{\geq 2}^n$, where we recover Beilinson's theorem by taking~$n=0$.
Weighted projective lines were introduced by Geigle and Lenzing in \cite{GeigleLenzingWPCarisinginRepTh} and also possess so-called ``canonical'' tilting bundles.
We therefore obtain a derived equivalence~$D^b(\Pbb^1(\boldsymbol{\lambda} ; \mathbf{w})) \simeq D^b(C(\boldsymbol{\lambda} ; \mathbf{w}))$.
In general, the canonical algebra has global dimension $2$, however there is a hereditary abelian subcategory
of regular representations that corresponds to the category of torsion sheaves on the weighted projective line. 
We even obtain an isomorphism of algebras between the CoHA of torsion sheaves on~$\Pbb^1(\boldsymbol{\lambda}; \mathbf{w})$ and the CoHA of regular representations of~$C(\boldsymbol{\lambda}; \mathbf{w})$.


If all the weights are equal to~$2$, then the Euler form of this category is symmetric and we can study the CoHA as a~$\Z \times \N_0^{Q_0}$-graded algebra.
The main result of this paper is an isomorphism
\[
\Coha(\Tor(\Pbb^1(\lambda_1, \dots, \lambda_n ; 2, \dots, 2))) \cong \Pp_n ,
\]
where~$\Pp_n$ is an algebra defined in terms of generators and relations.


\subsection*{Structure of this Paper}\addcontentsline{toc}{subsection}{Structure of this Paper}
In \Cref{section1}, we recall the notions from \cite{meinhardt2015donaldsonthomasinvariantsvsintersection} to associate Donaldson--Thomas invariants to hereditary abelian categories and define an associated CoHA in order to make precise how equivalent categories induce isomorphic CoHAs.
We also introduce the ChowHA which is defined just like the CoHA, but uses Fulton's intersection theory \cite{FultonIntersectionTheory} instead of singular cohomology and allows us to use the localization exact sequence to obtain generators of Chow groups of open subsets of affine spaces, see \Cref{Chowha primitive part} and \Cref{Coha generating degrees}.
In \Cref{section2}, we give a quick introduction to weighted projective lines and canonical algebras and prove a version of the equivalence~$\Tor(\Pbb^1(\boldsymbol{\lambda}; \mathbf{w})) = \Reg(C(\boldsymbol{\lambda}; \mathbf{w}))$ which yields isomorphic stacks of objects.
In \Cref{section3}, we provide a cell decomposition of these stacks (in the case~$\mathbf{w} = (2,\dots, 2)$), which shows that the cohomology thereof is given by their Chow rings, i.e., the cycle map from the ChowHA to the CoHA of these categories 
is an isomorphism of algebras.

There are natural functors between the categories of regular representations for different weight data that induce morphisms of quotient stacks and thus of the cohomology thereof.
This will induce a map of graded vector spaces between these CoHAs.
We show in \Cref{section4}, that in our setting this map will be an algebra homomorphism.

In \Cref{section5}, we define our new class of algebras~$\Pp_n$ for~$n \in \N_0$ in terms of generators and relations and in \Cref{section6} we consider a set of elements of~$\Coha(\Tor(\Pbb^1(2^n)))$ for~$n=0,\, 1, \, 2$ and show that these satisfy the relations of the algebras considered in \Cref{section5}.
Finally, we show in \Cref{section7} that these relations lift to all~$n$ and that the induced map
\[
\Pp_n \to \Coha(\Tor(\Pbb^1(\boldsymbol{\lambda}; 2^n)))
\]
is an isomorphism.

\subsection*{Notation} \addcontentsline{toc}{subsection}{Notation}

A quiver~$Q$ is a finite oriented graph and we denote the associated set of vertices by~$Q_0$.
We work with categories of representations of finite dimension over the complex numbers~$\C$.
We write~$\dimvectd = (d_i)_{i\in Q_0} \in \N_0^{Q_0}$ for a dimension vector of~$Q$.
All schemes and stacks are (locally) of finite type over~$\C$.
A variety is a separated reduced scheme of finite type. 
We write~$\pt = \Spec \C$.
Cohomology is taken with rational coefficients.
Cohomology of a quotient stack~$[X/G]$ is defined as the~$G$-equivariant cohomology of~$X$.

\section*{Acknowledgments} \addcontentsline{toc}{section}{Acknowledgments}

I want to thank my advisor Markus Reineke for introducing me to the wonderful world of quivers and CoHAs, for the many discussions on mathematics, and for his great support throughout my time in Bochum.
I also want to thank Olivier Schiffmann for hosting me for a week in Paris, where he introduced me to weighted projective lines and canonical algebras.
It was his idea that the algebras~$\Pp_n$ could be related to weighted projective lines in the first place.
This work was supported by the DFG CRC-TRR 191 “Symplectic structures in geometry, algebra and dynamics” (281071066).

\section{CoHAs and ChowHAs for Categories of Homological Dimension One}\label{section1}
In this section we recall the definition of the cohomological Hall algebra for hereditary abelian categories.
We do this by using the formalism of Donaldson--Thomas invariants for such categories developed in \cite{meinhardt2015donaldsonthomasinvariantsvsintersection}.

\subsection{Moduli of Objects in Abelian Categories and Donaldson--Thomas Invariants}\label{Subsection Axiomatics}
We give a brief summary of \cite[Section~3]{meinhardt2015donaldsonthomasinvariantsvsintersection}.
We want to start out with a~$\C$-linear abelian category~$\mathcal{A}_\C$ of global dimension~$1$ -- such as representations of a quiver or coherent sheaves on a smooth projective curve -- and take its moduli of objects.
However, there might not be an intrinsic way to do this if one is just given the abelian category itself. 
%
%
What we need to do is to consider~$S$-valued points of the stack of objects of an abelian category~$\mathcal{A}_\C$ instead of just its groupoid of objects.
We briefly recall the axiomatics from \cite{meinhardt2015donaldsonthomasinvariantsvsintersection} which guarantees such a framework.
\begin{itemize}
	\item[(1)] Existence of a moduli theory, i.e., an extension of~$\mathcal{A}_\C$ to a functor
	\[
	\mathcal{A}\colon \Sch_\C \to \{\text{exact categories}\},
	\]
	\item[(2)] Existence of a good moduli stack~$\Mf = \coprod_{\dimvectd\in \N_0^{\oplus I}} \Mf_\dimvectd$ of objects in~$\mathcal{A}$, where the component~$\Mf_\dimvectd$ is a global quotient stack $ \Mf_\dimvectd = [R_\dimvectd / G_\dimvectd]$ for some variety~$R_\dimvectd$ and~$G_\dimvectd = \prod_{i \in I} \GL_{d_i}(\C)$,
	\item[(3)] Existence of a faithful fiber functor~$\omega \colon \mathcal{A} \to \Vect^I$ for some finite indexing set~$I$,
	\item[(4)] Existence of good GIT-quotients~$\mathcal{M}$, i.e., a GIT-construction of course moduli spaces~$\mathcal{M}_\dimvectd$ of~$\Mf_\dimvectd$,
	\item[(5)] Representability and properness of the universal Grassmannian, i.e., a good moduli stack~$\mathfrak{Exact} = \coprod_{\dimvectd,\dimvecte} [R_{\dimvectd,\dimvecte} / G_{\dimvectd,\dimvecte}]$ of extensions in~$\mathcal{A}$, where~$R_{\dimvectd,\dimvecte}$ is a variety and~$G_{\dimvectd,\dimvecte}\subseteq G_{\dimvectd+\dimvecte}$ is the obvious parabolic, such that the morphism~$\pi_2\colon \mathfrak{Exact}\to \Mf$ mapping a short exact sequence to its middle term is proper,
	\item[(6)] Existence of a good deformation theory (which we make no reference to in this paper),
	\item[(7)] The numbers~$\langle E, F \rangle \coloneqq \dim_\C \Hom_{\mathcal{A}_\C} (E,F) - \dim_\C\Ext^1_{\mathcal{A}_\C}(E,F)$ are constant on the product~$\Mf_\dimvectd \times \Mf_\dimvecte$ and the variety~$R_\dimvectd$ is smooth,
	\item[(8)] The pairing~$\langle - , - \rangle$ is symmetric.
\end{itemize}

We also write~$\langle \dimvectd , \dimvecte \rangle$ for~$\langle E , F \rangle$ if~$E \in \Mf_\dimvectd$ and~$F \in \Mf_\dimvecte$.
Note that~$\dim \Mf_\dimvectd = - \langle \dimvectd , \dimvectd \rangle$ follows from these axioms by \cite[Corollary~3.24 and Proposition~3.25]{meinhardt2015donaldsonthomasinvariantsvsintersection}.

Note that we do not assume the varieties~$R_\dimvectd$ to be connected.

\begin{ex}\label{Ex:QuiverNOstabilityAssumptions}
	For a quiver~$Q$ (without relations) the space~$R_\dimvectd$ is equal to the space of based representations of~$Q$ with dimension vector~$\dimvectd$, i.e.,~$R_\dimvectd=R_\dimvectd(Q) \coloneqq \prod_{i \xrightarrow{\alpha} j } \mathbb{A}^{d_j\times d_i}$, with~$G_\dimvectd= \prod_{i \in Q_0}\GL_{d_i}(\C)$ acting via base change.
	The space~$R_{\dimvectd,\dimvecte}$ is the closed subvariety of~$R_{\dimvectd + \dimvecte}$ of upper block triangular matrices and~$G_{\dimvectd , \dimvecte}$ is the corresponding parabolic in~$G_{\dimvectd + \dimvecte}$.
	The category~$\Rep(Q)$ always satisfies assumptions~(1)--(7), while assumption~(8) is equivalent to the quiver being symmetric.
\end{ex}

\begin{ex}\label{Ex:L1Space}
	For the~$1$-loop quiver we have that~$\N_0^{Q_0} = \N_0$ and so
	\begin{align*}
		R_d & = \C^{d\times d}, \quad 		G_d = \GL_d(\C),  \\
		R_{d,e} & = \left\{ \twoByTwoMatrix{A}{B}{0}{C} \in \C^{(d+e) \times (d+e)} \ \middle| \ A \in \C^{d\times d}, \, B\in\C^{d \times e}, \, C \in \C^{e\times e} \right\}, \, \text{and} \\
		G_{d,e} & = \left\{ \twoByTwoMatrix{A}{B}{0}{C} \in \GL_{d+e}(\C) \ \middle| \ A \in \GL_d(\C), \, B\in\C^{d \times e}, \, C \in \GL_e(\C) \right\}.		
	\end{align*}
\end{ex}

\begin{ex}
	For a quiver~$Q$ with relations~$I$ the space~$R_\dimvectd$ is the closed subscheme of~$R_\dimvectd(Q)$ given by the relations.
	The category~$\Rep(Q; I)$ of representations satisfying the relations~$I$ always satisfies assumptions~(1)--(6) (except that~$R_\dimvectd$ could be non-reduced).
\end{ex}

\begin{ex}
	Given a quiver~$Q$ together with a stability function~$\theta\colon \Z^{Q_0} \to \Z$ and slope~$\mu_0 \in \Q$, we consider the hereditary category~$\Rep^{\tsst, \mu_0}(Q)$ of $\theta$-semistable representations with slope $\mu_0$, see \cite{ReinekeModuliOfRepresentationsOfQuivers}.
	Then the stack~$\Mf_\dimvectd$ satisfies~$\Mf_\dimvectd = [ R_\dimvectd / G_\dimvectd ]$ where the variety~$R_\dimvectd = R_\dimvectd^\tsst(Q) \subseteq R_\dimvectd(Q)$ is the open subscheme of semistable representations.
	The category~$\Rep^{\tsst, \mu_0}(Q)$ satisfies assumptions~(1)--(7), while assumption~(8) is equivalent to the form~$\langle - , - \rangle$ being symmetric on dimension vectors of~$\theta$-slope~$\mu_0$. 
\end{ex}

\begin{rk}
	It follows from assumptions~(1)--(3) that~$\Mf_0 = \pt$.
\end{rk}

These assumptions allow us to define the Donaldson--Thomas invariants of~$\mathcal{A}$ as an element of the Grothendieck group of (polarizable) mixed Hodge structures (see \cite{PSMHS} for generalities on mixed Hodge structures).
We write~$\MHM(X)$ for the category of mixed Hodge modules on~$X$,~$\MHS = \MHM(\pt)$ for the category of polarizable mixed Hodge structures, and~$\mathbb{L} = [H^\bullet_c(\mathbb{A}^1)] \in K_0(\MHS)$ for the class of the dimension~$1$, weight~$2$ pure Hodge structure. 

To define DT-invariants however, we need~$\lambda$-ring structures on these Grothendieck rings.


\begin{rk}
	For a commutative monoid~$(\mathcal{M}, \oplus)$ in~$\Sch_\C$ such that~$\oplus$ is a finite morphism, we have that~$K_0(\MHM(\mathcal{M}))$ has the structure of a (special)~$\lambda$-ring by \cite[Proposition~4.3]{meinhardt2015donaldsonthomasinvariantsvsintersection}.
	If an element~$M$ of the completion
	\[
	\hat{K}_0(\MHM(\mathcal{M})) = \prod_{m \in \pi_0(\mathcal{M})} K_0(\MHM(\mathcal{M}_m))
	\]
	has support away from the~$0$-component~$\mathcal{M}_0$, the total symmetric power
	\[
	\Sym(M) \coloneqq \bigoplus_{n \geq 0} \Sym^n(M)
	\]
	is well-defined.	
	We will consider the monoid~$\N_0^I$ associated to the abelian category~$\mathcal{A}$.
	We have
	\[
	\hat{K}_0(\MHM(\N_0^I)) = \prod_{i \in \N_0^I} K_0(\MHM(\pt)) = K_0(\MHS) [\![t_i : i \in I]\!]
	\]
	and we consider the extended rings
	\begin{gather*}
		K_0(\MHS)[\mathbb{L}^{-1/2}], \, K_0(\MHS)[\mathbb{L}^{-1/2} , (\mathbb{L}^N - 1 )^{-1}: N \in \Z_{\geq 1}] , \, \text{and} \\
		\hat{K}_0(\MHM(\N_0^I)) [\mathbb{L}^{-1/2} , (\mathbb{L}^N - 1 )^{-1}: N \in \Z_{\geq 1}],
	\end{gather*}
	where
	\[
	\hat{K}_0(\MHM(\N_0^I))[\mathbb{L}^{-1/2} , (\mathbb{L}^N - 1 )^{-1}: N \in \Z_{\geq 1}] =  K_0(\MHS)[\mathbb{L}^{-1/2} , (\mathbb{L}^N - 1 )^{-1}] [\![t_i : i \in I]\!]
	\]
	is a completion of~${K}_0(\MHM(\N_0^I))[\mathbb{L}^{-1/2} , (\mathbb{L}^N - 1 )^{-1}: N \in \Z_{\geq 1}]$.
	All these rings inherit a~$\lambda$-ring structure from the~$\lambda$-ring structure on~$K_0(\MHM(\N_0^{I})$.
	For details see \cite[Section~4]{meinhardt2015donaldsonthomasinvariantsvsintersection}, \cite[Section~3.3]{mozgovoy2024intersectioncohomologymodulispaces}, or \cite{LaurentiuSaitoSchuermannSymmProdOfMHM,LaurentiuSchuermannTwistedGeneraOfSymmProds}.
\end{rk}

\begin{mydef}[{\cite[Definition~5.2]{meinhardt2015donaldsonthomasinvariantsvsintersection}}] \label{Def:DT}
	Given a category~$\mathcal{A}$ satisfying the assumptions~(1)--(8) and a dimension vector~$\dimvectd \in \N_0^I$, we define the Donaldson--Thomas invariant~$\mathrm{DT}_\dimvectd \in K_0(\MHS)[\mathbb{L}^{-1/2}]$ by the equation
	\[
	\omega_! \mathcal{IC}_\Mf(\Q) = \Sym\left(\frac{1}{ \mathbb{L}^{1/2} - \mathbb{L}^{-1/2} } \sum_{\dimvectd\in \N_0^I} \mathrm{DT}_\dimvectd t^\dimvectd \right)
	\]
	in the~$\lambda$-ring~$K_0(\MHS)[\mathbb{L}^{-1/2} , (\mathbb{L}^N - 1 )^{-1}] [\![t_i : i \in I]\!]$, where~$\omega \colon \Mf \to \coprod_{\N_0^I} \pt$ maps the component~$\Mf_\dimvectd$ to the point labeled~$\dimvectd$.	
	\mycomment{
		\[
		\hat{K}_0(\MHM(\N_0^I))[\mathbb{L}^{-1/2} , (\mathbb{L}^N - 1 )^{-1}: N \in \Z_{\geq 1}] =  K_0(\MHS)[\mathbb{L}^{-1/2} , (\mathbb{L}^N - 1 )^{-1}] [\![t_i : i \in I]\!],
		\]}
\end{mydef}
\begin{rk}
	We can deduce the existence of Donaldson--Thomas invariants as elements of the ring~$K_0(\MHS)[\mathbb{L}^{-1/2} , (\mathbb{L}^N - 1 )^{-1} : N \in \Z_{\geq 1}]$ from generalities of~$\lambda$-rings.
	That the element actually lives in~$K_0(\MHS)[\mathbb{L}^{-1/2}]$  (or more precisely in the image of the map~$K_0(\MHS)[\mathbb{L}^{-1/2}] \to K_0(\MHS)[\mathbb{L}^{-1/2} , (\mathbb{L}^N - 1 )^{-1}]$) is the integrality conjecture \cite[Corollary~6.8]{meinhardt2015donaldsonthomasinvariantsvsintersection}.
	For integrality, we need assumption~(8).
\end{rk}

\begin{mydef}
	The Deligne--Hodge--Euler polynomial~$\chi_E \colon K_0(\MHS) \to \Z[u^{\pm 1},v^{\pm 1}]$ is given by~$\chi_E(H) = \sum_{i,j} h^{i,j}(H) u^i v^j$ for a mixed Hodge structure~$H$.
	It extends to a map~$\chi_E \colon K_0(\MHS)[\mathbb{L}^{-1/2} , (\mathbb{L}^N - 1 )^{-1}]  \to \Z[u^{\pm 1}, v^{\pm 1}] [(uv)^{-1/2},((uv)^N-1)^{-1}]$ by mapping~$\mathbb{L}^{1/2}$ to~$(uv)^{1/2}$.
	
	Using this, we can also define the quantum DT-invariants of the category~$\mathcal{A}$ by
	\[
	\mathrm{DT}^{\mathrm{quant}}_\dimvectd(q) \coloneqq \chi_E(\mathrm{DT}_\dimvectd)|_{u= v = (uv)^{1/2} = - q}\in \Z [q^{\pm 1}].
	\]
\end{mydef}

\begin{rk} \label{BettiNumbersFromK0(MHS)}
	For a class $[H] = [H_0] - [H_1] + [H_2] - \dots \pm [H_n] \in K_0(\MHS)$ with~$H_i$ pure of weight~$i$, we have that 
	\[
	\chi_E([H])|_{u= v = (uv)^{1/2} = - q} = \sum_{i=0}^n \dim (H_i) q^i \in \N_0[q],
	\]
	where~$\dim H_i$ is the dimension of the underlying vector space.
	It follows that the Betti numbers of pure varieties can be recovered from the class of its cohomology in~$K_0(\MHS)$.
\end{rk}


\subsection{Convolution Product and Cohomological Hall Algebras}

We now categorify these DT-invariants and define the cohomological Hall algebra associated to a category~$\mathcal{A}$ satisfying assumption~(1)--(7).

We are given a diagram
\[\begin{tikzcd}
	& {[R_{\dimvectd , \dimvecte}/G_{\dimvectd , \dimvecte}]} \\
	{[R_\dimvectd/G_\dimvectd ] \times [ R_\dimvecte/G_\dimvecte]} && {[R_{\dimvectd+\dimvecte}/G_{\dimvectd+\dimvecte}]},
	\arrow["{\mathrm{pr}_1\times \mathrm{pr}_3}"', from=1-2, to=2-1]
	\arrow["{\mathrm{pr}_2}", from=1-2, to=2-3]
\end{tikzcd}\]
where~$\mathrm{pr}_2$ is proper.
We can therefore define a product on the~$\Z \times \N_0^{\oplus I}$-graded vector space
\[
\Coha(\mathcal{A}) \coloneqq \bigoplus_{\dimvectd\in \N_0^{\oplus I}} H^\bullet([R_\dimvectd/G_\dimvectd]; \Q^\mathrm{vir}) = \bigoplus_{\dimvectd\in \N_0^{\oplus I}} H^\bullet_{G_\dimvectd}(R_\dimvectd;\Q^\mathrm{vir}) 
\]
via the composition
\mycomment{\[
	H^\bullet([R_\dimvectd/G_\dimvectd])\otimes H^\bullet([R_\dimvecte/G_\dimvecte]) \to H^\bullet([R_\dimvectd/G_\dimvectd]\times [R_\dimvecte/G_\dimvecte])\xrightarrow{(\mathrm{pr}_1\times \mathrm{pr}_3)^*} H^\bullet( [ R_{\dimvectd , \dimvecte}/G_{\dimvectd , \dimvecte} ] ) \xrightarrow{(\mathrm{pr}_2)_*} H^{\bullet - 2 \langle \dimvectd,\dimvecte \rangle}( [ R_{\dimvectd+\dimvecte} / G_{\dimvectd+\dimvecte} ] ),
	\]}
\[\begin{tikzcd}
	{	H^\bullet([R_\dimvectd / G_\dimvectd])\otimes H^\bullet([R_\dimvecte / G_\dimvecte])} & {H^\bullet([R_\dimvectd / G_\dimvectd]\times [R_\dimvecte / G_\dimvecte])} \\
	{H^\bullet( [ R_{\dimvectd , \dimvecte} / G_{\dimvectd , \dimvecte} ] ) } & { H^{\bullet - 2 \langle \dimvectd,\dimvecte \rangle}( [ R_{\dimvectd+\dimvecte} / G_{\dimvectd+\dimvecte} ] ),}
	\arrow["{\mathrm{K\ddot{u}n}}", from=1-1, to=1-2]
	\arrow[out=-30, in=150,"{(\mathrm{pr}_1\times \mathrm{pr}_3)^*}"{description}, from=1-2, to=2-1]
	\arrow["{(\mathrm{pr}_2)_*}"', from=2-1, to=2-2]
\end{tikzcd}\]

yielding an~$\N_0^{\oplus I}$-graded associative unitial algebra as in \cite[Section~2]{KSCoHADef}, where 
\[
H^\bullet([R_\dimvectd/G_\dimvectd]; \Q^\mathrm{vir}) = H^{\bullet - \langle \dimvectd, \dimvectd\rangle}( [R_\dimvectd/G_\dimvectd] ; \Q)({-\langle \dimvectd , \dimvectd \rangle}/{2})
\]
denotes the usual cohomological shift which guarantees that for smooth~$X$ Poincaré duality becomes~$H^k(X; \Q^\mathrm{vir}) = H^{-k}_c(X; \Q^\mathrm{vir})^\vee$.

\begin{rk}
	The multiplication maps~$H^i(\Mf_\dimvectd;\Q^\mathrm{vir})\otimes H^j(\Mf_\dimvecte; \Q^\mathrm{vir})$ into the cohomology group~$H^{k}(\Mf_{\dimvectd+\dimvecte}; \Q^\mathrm{vir})$ with
	\[
	k = i - \langle \dimvectd , \dimvectd \rangle + j - \langle \dimvecte , \dimvecte \rangle - 2 \langle \dimvectd , \dimvecte \rangle + \langle \dimvectd + \dimvecte , \dimvectd + \dimvecte \rangle = i + j + \langle \dimvecte , \dimvectd \rangle - \langle \dimvectd , \dimvecte \rangle. 
	\]
	Under the symmetry assumption~(8) from \cite{meinhardt2015donaldsonthomasinvariantsvsintersection}, i.e.,~$\langle \dimvectd, \dimvecte \rangle = \langle \dimvecte , \dimvectd \rangle$,~$\Coha(\mathcal{A})$ becomes a~$\Z\times \N_0^{\oplus I}$-graded algebra with finite dimensional graded pieces.
\end{rk}

%
%

The following is the reason why we call the CoHA a ``categorification'' of the (quantum) DT-invariants in the case that~$\mathcal{A}$ satisfies assumptions~(1)--(8).
\begin{prop}\label{Poincare series of Coha = Sym(DT}
	If the cohomology of the stack~$\Mf_\dimvectd$ is pure for all dimension vectors~$\dimvectd$, then we can recover the Poincaré series of the algebra~$\Coha(\mathcal{A})$ from its class in the ring~$K_0(\MHS)[\mathbb{L}^{-1/2} , (\mathbb{L}^N - 1 )^{-1}] [\![t_i : i \in I]\!]$.
	Explicitly, we have
	\[
	P_{q^{-1},t}(\Coha(\mathcal{A})) = \chi_E(\Sym\left(\frac{1}{ \mathbb{L}^{1/2} - \mathbb{L}^{1/2} } \sum_{\dimvectd\in \N_0^I} \mathrm{DT}_\dimvectd t^\dimvectd \right))|_{u = v = - q}.
	\]
\end{prop}
\begin{proof}
	The assertions follow from \Cref{BettiNumbersFromK0(MHS)}.
	Note that the~$q^{-1}$ comes from the fact that we need to reverse the cohomological grading, because in the definition of the DT-invariants we took~$\omega_! \mathcal{IC}_\Mf(\Q)$, which corresponds to cohomology with compact support.
\end{proof}

\begin{rk}
	The category~$\Rep^{\tsst, \mu_0}(Q)$ of semistable representations of a quiver satisfies the assumption of \Cref{Poincare series of Coha = Sym(DT} by \cite[Theorem~5.1]{FRChowHa}.
\end{rk}

\subsection{Chow Hall Algebra}

The morphism~$ [ R_{\dimvectd , \dimvecte} / G_{\dimvectd , \dimvecte} ] \to [ R_{\dimvectd} / G_\dimvectd ] \times [ R_\dimvecte / G_\dimvecte ]$ decomposes as 
\[
[R_{\dimvectd , \dimvecte}/G_{\dimvectd , \dimvecte}] \to [R_{\dimvectd , \dimvecte}/(G_\dimvectd\times G_\dimvecte)] \to [(R_\dimvectd\times R_\dimvecte)/(G_\dimvectd\times G_\dimvecte)] = [R_\dimvectd/G_\dimvectd] \times [R_\dimvecte/G_\dimvecte] .
\]
The first map is a fibration with fiber~$G_{\dimvectd , \dimvecte}/(G_\dimvectd\times G_\dimvecte) \cong \mathbb{A}^{N}$ and the second is a vector bundle by \cite[Corollary~3.28]{meinhardt2015donaldsonthomasinvariantsvsintersection}.
The composition is therefore smooth.
It follows that the definition of the CoHA and its convolution product can be generalized from cohomology (or rather Borel--Moore homology) to arbitrary oriented Borel--Moore homology theories, see \cite[Chapter~5]{LevineMorelAlgCobordism}.

We now use Chow groups -- another oriented Borel--Moore homology theory -- to define the analogue of the Chow Hall algebra (ChowHA) from \cite{FRChowHa} for hereditary categories instead of for quivers.
For a quotient stack~$ [ X / G ]$, we write
\[
A_\bullet( [ X / G ] ) = A^G_{\bullet + \dim G}( X ; \Q) \quad \text{ and } \quad A^\bullet([X/G]) = A^\bullet_G(X;\Q)
\]
for the equivariant Chow groups/rings with rational coefficients as defined in \cite{EquivariantIntersectionTheory}.

As with cohomology, we define an associative algebra structure on
\[
\Chowha(\mathcal{A})\coloneqq \bigoplus_{\dimvectd\in \N_0^{\oplus I}} A^\bullet([R_\dimvectd/G_\dimvectd] ; \Q^\mathrm{vir}).
\]
We have an algebra morphism~$\Chowha(\mathcal{A})\to \Coha(\mathcal{A})$ which is compatible with the~$\N_0^{\oplus I}$-grading and doubles the cohomological degree.


\subsection{The Simple ChowHA}

Let~$\Mf_\dimvectd^\mathrm{simp}\subseteq \Mf_\dimvectd$ be the stack classifying simple objects.
It is an open substack because it is the complement of the (finitely many) images of the proper morphisms~$\Mf_{\dimvectd_1\!,\dimvectd_2} \to \Mf_\dimvectd$, where~$\dimvectd=\dimvectd_1+\dimvectd_2$ and~$\dimvectd_1, \dimvectd_2\neq 0$.
We equip the graded vector space
\[
\Chowha^\simp(\mathcal{A})\coloneqq A^\bullet(\Mf_0)[0]\oplus \bigoplus_{\dimvectd\neq 0}A^\bullet(\Mf^\simp_\dimvectd) =  \Q[0] \oplus \bigoplus_{\dimvectd\neq 0}A^\bullet(\Mf^\simp_\dimvectd)
\]
with the trivial product, i.e.,~$f*g = 0$ for homogeneous elements of degree~$\neq 0$.

\begin{prop}[{{{\cite[Theorem~9.1]{FRChowHa}}}}]\label{Chowha primitive part}
	Assume that for all~$\dimvectd_1,  \dimvectd_2 \in \N_0^I$ the Künneth morphism~$A^\bullet(\Mf_{\dimvectd_1})\otimes_\Q A^\bullet(\Mf_{\dimvectd_2}) \to A^\bullet(\Mf_{\dimvectd_1}\times \Mf_{\dimvectd_2})$ is surjective.
	Then the map
	\[
	\Chowha(\mathcal{A})\to \Chowha^\simp(\mathcal{A})
	\]
	is a surjective~$\mycomment{A^\bullet(\Mf_0)=}\Q$-algebra morphism with kernel given by the square of the augmentation ideal, i.e.,
	\[
	\Chowha^\simp(\mathcal{A}) \cong \Chowha(\mathcal{A})/(\Chowha(\mathcal{A})_+* \Chowha(\mathcal{A})_+),
	\]
	where~$\Chowha(\mathcal{A})_+ \coloneqq \bigoplus_{\dimvectd\neq 0} \Chowha(\mathcal{A})_\dimvectd$. 
	It follows that elements in~$\Chowha(\mathcal{A})$ that map to a~$\Q$-basis of~$\Chowha^\simp(\mathcal{A})$ will generate the ChowHA as a~$\Q$-algebra.
\end{prop}

\section{Weighted Projective Line and the Canonical Algebra}\label{section2}

We collect some facts from the theory of weighted projective lines as developed by Geigle--Lenzing in \cite{GeigleLenzingWPCarisinginRepTh} -- for a comprehensive introduction also see \cite{chen2010introductioncoherentsheavesweighted} -- and of canonical algebras which were introduced and studied by Ringel in \cite{RingelSLN1099,RingelCanonicalAlgs}.

\subsection{Weighted Projective Line}
We recall the definition of the weighted projective line given in \cite[Section~2.2]{ChanLernerModuliOfSerreStableReps}.
For us a weighted projective line will be a certain stacky curve, i.e., a proper smooth connected Deligne--Mumford stack of dimension~$1$ with a dense open subscheme.
In our case, it will be glued from stacks of the form~$[U/\mu_p]$ for~$U$ a~$1$-dimensional variety and~$\mu_p$ the cyclic group of~$p$-th roots of unity.

Starting with the projective line~$\Pbb^1\coloneqq \Pbb^1_\C$, choose~$n$ distinct closed points~$\lambda_1,\dots, \lambda_n \in \Pbb^1$ and~$n$ positive integers~$w_1, \dots, w_n$.
For this data we define the weighted projective line~$\Pbb^1(\boldsymbol{\lambda}; \mathbf{w}) \coloneqq \Pbb^1(\lambda_1,\dots, \lambda_n; w_1, \dots, w_n)$ (or also~$\Pbb^1(\mathbf{w}) = \Pbb^1(w_1,\dots, w_n)$ if the choice of points~$\lambda_i$ is implicit) together with a morphism~$\Pbb^1(\boldsymbol{\lambda}; \mathbf{w}) \to \Pbb^1$ as follows:
Take the open subscheme~$U_i\coloneqq \Pbb^1\setminus\{\lambda_1, \dots, \lambda_{i-1}, \lambda_{i+1}, \dots \lambda_n\}$ (which is affine for~$n\geq 2$) and set~$\widetilde{U}_i\coloneqq \Spec \mathcal{O}_{U_i}[s]/(s^{w_i}-(t-\lambda_i))$ where~$t$ is the local coordinate of~$\Pbb^1$.
We have a ramified~$w_i$-cover~$\widetilde{U}_i\to U_i$ which is totally ramified above~$\lambda_i$ and unramified everywhere else.
Now we glue the quotient stacks~$[\widetilde{U}_i/\mu_{w_i}]$ together along~$\Pbb^1\setminus \{ \lambda_1, \dots, \lambda_n \}$.
This is compatible with the morphisms~$[\widetilde{U}_i/\mu_{w_i}]\to U_i$ yielding a stack~$\Pbb^1(\lambda_1,\dots, \lambda_n; w_1, \dots, w_n)$ together with a morphism to~$\Pbb^1$.
On the open subscheme~$\Pbb^1\setminus \{\lambda_1, \dots, \lambda_n \}$ this is an isomorphism.

Note that the stack~$\Pbb^1(\lambda_1,\dots, \lambda_n; w_1, \dots, w_n)$ is independent of the order of the~$\lambda_i$ and that it satisfies~$\Pbb^1(\lambda_1,\dots, \lambda_n; w_1, \dots, w_n)\cong \Pbb^1(\lambda_1,\dots, \lambda_n, \lambda_{n+1} ; w_1, \dots, w_n, 1)$.
This also allows us to define the weighted projective line in the case that~$n=0,\, 1$ in a compatible way with the above.
In particular, for~$n=0$ the weighted projective line is just~$\Pbb^1$.


We can now define coherent (and torsion) sheaves on the weighted projective line as in \cite{ChampsAlg} or \cite[Tag:~08KA]{stacks-project}.

\begin{rk}
	For a discussion on why this approach to coherent sheaves on weighted projective lines is equivalent to the classical one from Geigle--Lenzing, see \cite[Section~11]{ChanLernerModuliOfSerreStableReps}.
	Roughly, this equivalence is proven in the following way:
	The classical approach in terms of quotient categories of the form~$\Gr(R)/\mathrm{tors}$, where~$R$ is a suitable~2-dimensional graded ring, yields an equivalence 
	\[
	\Gr(R)/\mathrm{tors} \simeq \qcoh \StProj(R),
	\]
	where~$\StProj(R)$ is the stacky projective scheme and so the classical weighted projective line of Geigle--Lenzing can be defined as~$\StProj(R)$.
	Similarly to the usual patching of the~$\Proj$ construction from affine opens, the stacky projective scheme can be glued together from open substacks of the from~$[U / G]$.
	This gives the equivalence with the above definition.
\end{rk}

As discussed in \Cref{Subsection Axiomatics}, we need not only the (hereditary) abelian category of torsion sheaves to talk about its stack of objects (and extensions), but we need exact categories of torsion sheaves on~$\Pbb^1(\boldsymbol{\lambda}; \mathbf{w})\times S$ for any test scheme~$S$.
\begin{mydef}
	A torsion sheaf on~$\Pbb^1(\boldsymbol{\lambda}; \mathbf{w})\times S$ is a coherent sheaf~$M$ on~$\Pbb^1(\boldsymbol{\lambda}; \mathbf{w})\times S$ that is flat over~$S$ such that the map~$\supp(M)\to S$ is finite.
	We denote this category by~$\Tor(\Pbb^1(\boldsymbol{\lambda}; \mathbf{w}))_S$ and endow it with the standard exact structure (as an extension closed subcategory of the abelian category~$\Coh(\Pbb^1(\boldsymbol{\lambda}; \mathbf{w})\times S)$).
\end{mydef}

\subsection{Canonical Algebra}
We now introduce the canonical algebra~$C(\boldsymbol{\lambda}; \mathbf{w})$.
It will turn out that
\[
D^b(\Coh(\Pbb^1(\boldsymbol{\lambda}; \mathbf{w}))) \cong D^b(C(\boldsymbol{\lambda}; \mathbf{w})) \quad \text{and} \quad \Tor(\Pbb^1(\boldsymbol{\lambda}; \mathbf{w})) \cong \Reg(C(\boldsymbol{\lambda}; \mathbf{w})),
\]
where~$\Reg(C(\boldsymbol{\lambda}; \mathbf{w})$ is the abelian subcategory of~$\Rep(C(\boldsymbol{\lambda}; \mathbf{w}))$ given in \Cref{definitionRegular}.
\begin{mydef}
	For~$n$ distinct elements~$\lambda_1,\dots, \lambda_n \in \Pbb^1(\C)$ and weights~$w_1, \dots w_n\geq 1$, we define the canonical algebra~$C(\lambda_1 , \dots, \lambda_n ; w_1, \dots, w_n)$ as the quiver
	\[
	Q(\mathbf{w}) \coloneqq Q(C(\mathbf{w})) \coloneqq 
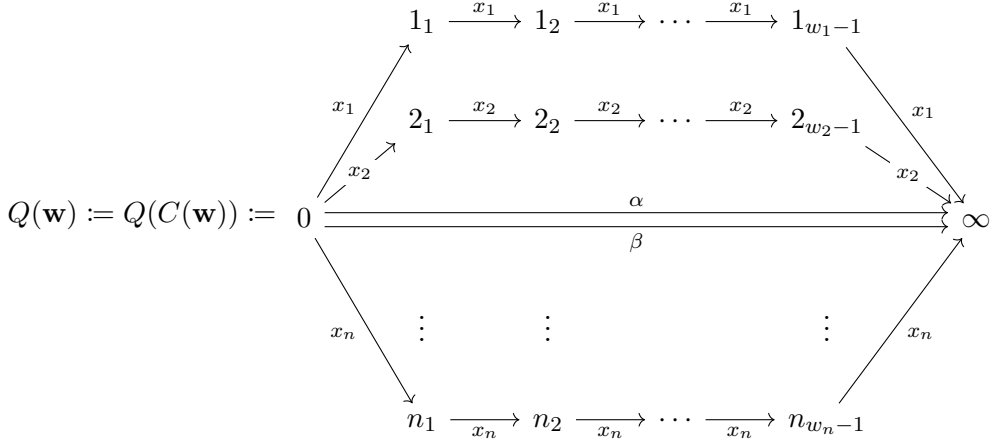
\begin{tikzcd}
		& {1_1} & {1_2} & \cdots & {1_{w_1-1}} \\
		& {2_1} & {2_2} & \cdots & {2_{w_2-1}} \\
		0 &&&&& \infty \\
		& \vdots & \vdots && \vdots \\
		& {n_1} & {n_2} & \cdots & {n_{w_n-1}}
		\arrow["{x_1}", from=1-2, to=1-3]
		\arrow["{x_1}", from=1-3, to=1-4]
		\arrow["{x_1}", from=1-4, to=1-5]
		\arrow["{x_1}", from=1-5, to=3-6]
		\arrow["{x_2}", from=2-2, to=2-3]
		\arrow["{x_2}", from=2-3, to=2-4]
		\arrow["{x_2}", from=2-4, to=2-5]
		\arrow["{x_2}"{description}, from=2-5, to=3-6]
		\arrow["{x_1}", from=3-1, to=1-2]
		\arrow["{x_2}"{description}, from=3-1, to=2-2]
		\arrow["\alpha", shift left=1, from=3-1, to=3-6]
		\arrow["\beta"', shift right=1, from=3-1, to=3-6]
		\arrow["{x_n}"', from=3-1, to=5-2]
		\arrow["{x_n}"', from=5-2, to=5-3]
		\arrow["{x_n}"', from=5-3, to=5-4]
		\arrow["{x_n}"', from=5-4, to=5-5]
		\arrow["{x_n}"', from=5-5, to=3-6]
	\end{tikzcd}\]
	subject to the relations~$x_i^{w_i} = \lambda_i^{(0)} \alpha +  \lambda_i^{(1)} \beta$, where~$\lambda_i = [\lambda_i^{(0)} : \lambda_i^{(i)} ] \in \Pbb^1(\C)$.
	This is independent of the chosen representative~$\lambda_i = [ \lambda_i^{(0)} : \lambda_i^{(2)}]$ up to isomorphism.
\end{mydef}

\begin{rk}\label{PGL2 on C(lambdaw)}
	Any element of~$\varphi\in \PGL_2(\C)$ induces an isomorphism
	\[
	C(\lambda_1, \dots, \lambda_n ; w_1, \dots, w_n) \cong C(\varphi(\lambda_1), \dots, \varphi(\lambda_n ); w_1,\dots, w_n).
	\]
\end{rk}

\begin{rk}\label{CanonicalAlgebraAdd1}
	We have~$C(\lambda_1, \dots, \lambda_n, \lambda_{n+1} ; w_1, \dots, w_n,1 )\cong C(\lambda_1, \dots, \lambda_n ; w_1, \dots, w_n)$.
	This allows us to always assume that~$w_i\geq 2$.
\end{rk}

\begin{ex}\label{CanAlgsForSmalln}
	\textbullet\hskip\labelsep For~$n=0$, the canonical algebra is the path algebra of the Kronecker quiver~$\begin{tikzcd}
		\bullet \rar[bend left] \rar[bend right] & \bullet
	\end{tikzcd}$.
	\begin{itemize}
		\item 	For~$n=1$,~$\lambda_1 = 0 = [0:1]$, and~$w_1=2$, the canonical algebra is the path algebra of the~$\tilde{A}_2$ quiver
		
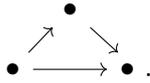
\[
		\begin{tikzcd}[column sep = tiny, row sep = small]
			& \bullet \drar & \\ \bullet \urar \ar{rr} &  & \bullet \, \, .
		\end{tikzcd}
		\]
		\item For~$n=2$,~$\lambda_1 = [0 : 1]$,~$\lambda_2 = \infty = [1 : 0 ]$, and~$w_1=w_2=2$, the canonical algebra is the path algebra of the~$\tilde{A}_3$ quiver
		\[\begin{tikzcd}[row sep = small]
			& \bullet \\
			\bullet && \bullet \, \, . \\
			& \bullet
			\arrow[from=1-2, to=2-3]
			\arrow[from=2-1, to=1-2]
			\arrow[from=2-1, to=3-2]
			\arrow[from=3-2, to=2-3]
		\end{tikzcd}\]
	\end{itemize}
\end{ex}

\begin{prop}[{\cite[Corollary~3.11]{GeigleLenzingWPCarisinginRepTh}}]
	The canonical algebra~$C(\boldsymbol{\lambda}; \mathbf{w})$ has global dimension~$\leq 2$.
\end{prop}


\begin{mydef}\label{definitionRegular}
	A representation~$M$ of~$C(\boldsymbol{\lambda}; \mathbf{w})$ is called regular if it is semistable of slope~$0$ with respect to the stability~$\dimvectd = (d_i)_{i\in Q_0(\mathbf{w})} \mapsto d_0-d_\infty$, i.e., if the dimension vector~$\dimvectd$ of~$M$ satisfies~$d_0= d_\infty$ and for every subrepresentation~$N$ of~$M$ of dimension vector~$\dimvecte$ we have~$e_0 \leq e_\infty$.
	We denote the full abelian subcategory of all regular representations by~$\Reg(C(\boldsymbol{\lambda}; \mathbf{w}))$.
\end{mydef}

\begin{rk}
	Regular representations can equivalently be defined as direct sums of indecomposable representations that are neither preprojective nor preinjective, see {\cite[Section~3 Corollary]{RingelCanonicalAlgs}}.
\end{rk}

The following is immediate from \Cref{definitionRegular}.

\begin{cor}\label{rep of C sst iff Kronecker sst}
	A representation~$M=(\alpha, \beta, A_1^{(1)}, \dots A_1^{(w_1)}, A_2^{(1)}, \dots, A_n^{(w_n)})$ of~$C(\boldsymbol{\lambda}; \mathbf{w})$ is regular if and only if~$(\alpha, \beta)$ is a regular representation of the~$2$-Kronecker quiver~$K_2$, i.e., a semistable slope~$0$ representation of~$K_2$.
\end{cor}

\begin{prop}[{\cite[Section~4 Corollary~2]{RingelCanonicalAlgs}}]
	The category of regular representations has global dimension~$1$.
\end{prop}

\begin{mydef}
	The Euler form~$\langle M, N\rangle$ of two representations~$M,N$ of~$C(\boldsymbol{\lambda}; \mathbf{w})$ is defined as
	\[
	\langle M, N \rangle \coloneqq \dim \Hom(M,N) - \dim \Ext^1(M,N) + \dim \Ext^2(M,N).
	\]
	For~$M,N$ regular this simplifies to~$\langle M, N \rangle = \dim \Hom(M,N) - \dim \Ext^1(M,N)$.
\end{mydef}

\begin{prop}[{{\cite[Proposition~3.13]{ASS1}}}] \label{EulerFormFromDimension}
	The Euler form~$\langle - , - \rangle$ of two representations~$M, N$ of~$C(\boldsymbol{\lambda}; \mathbf{w})$ only depends on their dimension vectors~$\dimvectd=\dim M$ and~$\dimvecte=\dim N$.
	We therefore also write~$\langle \dimvectd, \dimvecte \rangle \coloneqq \langle M, N \rangle$.
	Explicitly, it is given by 
	\[
	\langle \dimvectd , \dimvecte \rangle = \langle \dimvectd , \dimvecte \rangle_{Q(\mathbf{w})} + n d_0 e_\infty = \sum_{i\in Q_0(\mathbf{w})} d_i e_i - \sum_{i \to j} d_i e_j + n d_0 e_\infty.
	\]
	Here, the summand~$n d_0 e_\infty$ comes from the~$n$ relations in the algebra.
\end{prop}


\begin{cor}\label{Euler Form symmetric iff w=2^n}
	Let $w_i \geq 2$.
	The category~$\Reg(C(\boldsymbol{\lambda}; \mathbf{w}))$ has symmetric Euler form if and only if all~$w_i=2$.
	In this case we write~$\dimvectd= (d_0, d_1, \dots, d_n, d_0)$ for the dimension vector and we can simplify the Euler form to 
	\[
	\langle \dimvectd , \dimvecte \rangle = n d_0e_0+ \sum_{i=1}^n (d_i e_i- d_0 e_i - d_ie_0).
	\]
\end{cor}



\subsection{The Stack of Regular Representations}
We can now talk about the stack of objects of regular representations of a canonical algebra.
We construct the functor~$S\mapsto \Reg(C(\boldsymbol{\lambda}; \mathbf{w}))_S$ in such a way that this stack will be given by the disjoint union over all dimension vectors~$(d_0, d_1, \dots, d_n, d_0)$ of quotient stacks.
\begin{mydef}
	For a scheme~$S$ let the category of~$S$-valued regular representations be the category of finite locally free~$\mathcal{O}_S$-module representations of the quiver of~$C(\boldsymbol{\lambda}; \mathbf{w})$ that satisfy the defining relations and are semistable.
\end{mydef}
\begin{cor}\label{Mfd as stack}
	The stack of objects~$\Mf$ in~$\Reg(C(\boldsymbol{\lambda}; \mathbf{w}))$ is a disjoint union
	\[
	\Mf = \coprod_{\substack{\dimvectd\in \N_0^{Q_0( \mathbf{w} )} \! , \\ d_0 = d_\infty}} \Mf_\dimvectd ,
	\]
	where~$\Mf_\dimvectd$ is the quotient stack~$[R_\dimvectd / G_\dimvectd]$ with
	\[
	R_\dimvectd = \left\{ \left(\alpha, \beta, A_1^{(1)} \! , \dots, A_{1}^{(w_1)} \! , A_2^{(1)} \! , \dots, A_{n}^{(w_n)}\right) \ \middle| \  \begin{aligned}
		&\lambda^{(0)}_i \alpha +  \lambda^{(1)}_i \beta = A_{ i }^{( w_i )}\cdots A_i^{(1)} \text{ and}\\ &\text{$(\alpha, \beta)$ is a semistable representation}\\ &\text{of the quiver~$K_2$}
	\end{aligned} \right\}
	\]
	for~$\alpha, \, \beta \in \C^{d_\infty \! \times d_0}$ and~$(A_1^{(1)} \! , \dots, A_{n}^{ ( w_n ) }) \in \C^{d_{1_1} \! \times d_0} \times \C^{d_{1_2} \! \times d_{1_1}} \times \dots \times \C^{d_\infty \! \times d_{n_{w_n-1}}}$, and
	\[
	G_\dimvectd = \prod_{i\in Q_0( \mathbf{w} )}\GL_{d_i}(\C).
	\]
	
	Similarly, we have that~$\mathfrak{Exact}=\coprod_{ \dimvectd , \dimvecte } [ R_{ \dimvectd , \dimvecte }/G_{ \dimvectd , \dimvecte } ]$ with
	\[
	R_{ \dimvectd , \dimvecte } = \left\{ (\alpha, \dots, A_{n}^{(w_n)}) \in R_{ \dimvectd + \dimvecte } \ \middle| \  	\alpha, \dots , A_{n}^{(w_n)} \text{ are block upper triangular} \right\} \subseteq R_{ \dimvectd + \dimvecte }
	\]
	and
	\[
	G_{ \dimvectd , \dimvecte } = \left\{(g_i)\in G_{ \dimvectd + \dimvecte } \ \middle| \   \text{$g_i$ is block upper triangular} \right\}.
	\]
\end{cor}

\begin{prop}\label{Rd smooth + dim}
	The variety~$R_\dimvectd$ is a smooth irreducible quasi-projective variety of dimension~$-\langle \dimvectd , \dimvectd\rangle + \dim G_\dimvectd$. 
	Additionally, the closed immersion~$R_\dimvectd \mono R_\dimvectd^\sst(Q(\mathbf{w}))$ is regular.
\end{prop}
\begin{proof}
	The variety is irreducible of dimension~$-\langle \dimvectd , \dimvectd \rangle + \dim G_\dimvectd$ by \cite[3.6 and~3.7]{BobinskiGeomOfRegModulesOverCanAlgs}.
	It is smooth by \cite[Corollary~3.28]{meinhardt2015donaldsonthomasinvariantsvsintersection}.
	Regularity of the immersion follows as the codimension of~$R_\dimvectd$ in~$R_\dimvectd^\sst(Q(\mathbf{w}))$ is given by
	\[
	(-\langle \dimvectd , \dimvectd \rangle_{Q(\mathbf{w})} + \dim G_\dimvectd) - (-\langle \dimvectd , \dimvectd \rangle + \dim G_\dimvectd) = n\cdot d_0^2,
	\]
	which is equal to the number of equations used to define~$R_\dimvectd$ in~$R_\dimvectd^\sst(Q( \mathbf{w} ))$.
\end{proof}


\begin{prop}\label{Rde smooth + dim}
	The map~$R_{ \dimvectd , \dimvecte }\to R_\dimvectd \times R_\dimvecte$ is a vector bundle with fibers of dimension~$-\langle \dimvecte , \dimvectd \rangle + \dim G_{ \dimvectd , \dimvecte } - \dim G_\dimvectd - \dim G_\dimvecte$.
	It follows that~$R_{ \dimvectd , \dimvecte }$ is a smooth connected variety.
	Additionally, the closed immersion~$R_{ \dimvectd , \dimvecte }\mono R_{ \dimvectd , \dimvecte }^\sst(Q(\mathbf{w}))$ is regular.
\end{prop}
\begin{proof}
	By \cite[Corollary~3.28]{meinhardt2015donaldsonthomasinvariantsvsintersection},~$R_{ \dimvectd ,\dimvecte } \to R_\dimvectd \times R_\dimvecte$ is a vector bundle and the fiber dimension is calculated in \cite[Proposition~3.25]{meinhardt2015donaldsonthomasinvariantsvsintersection} to be
	$\sum_{i \in Q_0(\mathbf{w})} d_i e_i - \langle \dimvecte , \dimvectd \rangle$.
	The codimension of~$R_{ \dimvectd , \dimvecte }$ in~$R^\sst(Q( \mathbf{w} ))$ is thus equal to~$n (d_0^2+e_0^2+ d_0e_0)$, which also equals the number of equations defining~$R_{ \dimvectd , \dimvecte }$ in~$R_{ \dimvectd , \dimvecte }^\sst(Q( \mathbf{w} ))$.
\end{proof}

\begin{prop}\label{RegRepsSatisfyAssumptions}
	The category~$\Reg(C(\boldsymbol{\lambda}; \mathbf{w}))$ satisfies the assumptions (1)--(7) of \cite[Section~3]{meinhardt2015donaldsonthomasinvariantsvsintersection}, while assumption~(8) is equivalent to~$w_i \leq 2$ for~$i=1, \dots, n$.
\end{prop}
\begin{proof}
	It is clear from \cite[Section~3]{meinhardt2015donaldsonthomasinvariantsvsintersection} and \Cref{Mfd as stack} that for any~$w$ the category~$\Reg(C(\boldsymbol{\lambda}; \mathbf{w}))$ satisfies assumptions (1)--(6), while assumption (7) follows from \Cref{Rd smooth + dim}.
	By \Cref{Euler Form symmetric iff w=2^n}, the Euler form is symmetric if and only if~$w_i \leq 2$.
\end{proof}

\subsection{The Relationship between Canonical Algebras and Weighted Projective Lines}
The following is well-known, see for instance \cite{chen2010introductioncoherentsheavesweighted}, \cite[Section~4]{GeigleLenzingWPCarisinginRepTh}, or \cite{BKLExtremalPropsConCanAlgs}.
There is a tilting object~$T_\mathrm{can}$ 
in~$\Coh(\Pbb^1(\boldsymbol{\lambda}; \mathbf{w}))$ 
which induces an equivalence of categories~$R\Hom(T_\mathrm{can},-)\colon D^b(\Pbb^1(\boldsymbol{\lambda}; \mathbf{w})) \to D^b(C(\boldsymbol{\lambda}; \mathbf{w}))$, where~$C(\boldsymbol{\lambda}; \mathbf{w})=\End(T_\mathrm{can})^{\mathrm{op}}$ is the canonical algebra.
This equivalence restricts to an equivalence between torsion sheaves and regular/semistable representations. 

For our purposes, an equivalence of abelian categories is insufficient; we need isomorphic moduli-stacks of objects and extensions.
The following remedies this problem.

\begin{prop}\label{Tor(Pbb1) = Reg(C) stacky}
	The functors~$S\mapsto \Tor(\Pbb^1(\boldsymbol{\lambda}; \mathbf{w}))_S$ and~$S\mapsto \Reg(C(\boldsymbol{\lambda}; \mathbf{w}))_S$ are naturally isomorphic, inducing isomorphisms of stacks of objects/extensions.
\end{prop}
\begin{proof}[Proof Sketch]
	We need to show that for every test scheme~$S$ there is a natural equivalence of exact categories
	$\Tor(\Pbb^1(\boldsymbol{\lambda}; \mathbf{w}))_S\to \Reg(C(\boldsymbol{\lambda}; \mathbf{w}))_S = \Rep^{\sst}(C(\boldsymbol{\lambda}; \mathbf{w}))_S$, where
	\begin{align*}
		\Rep^{\sst}(C(\boldsymbol{\lambda}; \mathbf{w}))_S & = \left\{	\begin{tabular}{c}
			\text{locally free~$\mathcal{O}_S\otimes C(\boldsymbol{\lambda}; \mathbf{w})$-modules~$M$} \\
			\text{such that~$k(s)\otimes M$ is semistable for every~$s\in S$}
		\end{tabular} 	\right\} \\
		\intertext{and}
		\Tor(\Pbb^1(\boldsymbol{\lambda}; \mathbf{w}))_S & = \left\{	\begin{tabular}{c}
			\text{coherent sheaves~$M$ on~$S\times \Pbb^1(\boldsymbol{\lambda}; \mathbf{w})$ flat over~$S$} \\
			\text{such that~$\supp(M)\to S$ is finite}
		\end{tabular}	\right\}.
	\end{align*}
	Now, the functor~$(\mathrm{pr}_1)_*\shom(\mathrm{pr}_2^*(T_\mathrm{can}), -)$ does what we require, where~$\mathrm{pr}_1,\mathrm{pr}_2$ are the projections on~$S\times \Pbb^1(\boldsymbol{\lambda}; \mathbf{w})$. 
	
	For more details, compare this with the proofs of \cite[Theorem~6.1.1 and Theorem~6.1.2]{soibelman2014modulistackparabolicbundles}, where an equivalence between the stacks of certain vector bundles and of preinjective representations is given.
\end{proof}

By \Cref{RegRepsSatisfyAssumptions}, the category of regular representations (and thus also the category of torsion sheaves on a weighted projective line) satisfies the assumptions~(1)--(7) from \Cref{section1}.
We can therefore construct CoHAs and ChowHAs of the category~$\Tor(\Pbb^1(\boldsymbol{\lambda}; \mathbf{w})) = \Reg(C(\boldsymbol{\lambda}; \mathbf{w}))$.
As assumption~(8) is satisfied if and only if~$w_i =  2$ for all~$i= 1 , \dots, n$, we are especially interested in this case.

\begin{mydef}
	We will denote the CoHAs and ChowHAs of the categories of torsion sheaves on weighted projective lines/regular representations of canonical algebras in the case of a symmetric Euler form by
	\begin{align*}
		\Coha(\Pbb^1(2^n)) & \coloneqq \Coha(\Tor(\Pbb^1(\boldsymbol{\lambda}; 2^n))) = \Coha(\Reg(C(\boldsymbol{\lambda}; 2^n)))\quad  \text{and}\\
		\Chowha(\Pbb^1(2^n)) & \coloneqq \Chowha(\Tor(\Pbb^1(\boldsymbol{\lambda}; 2^n))) =\Chowha(\Reg(C(\boldsymbol{\lambda}; 2^n))) , 
	\end{align*}
	where we write~$(\boldsymbol{\lambda}; 2^n)$ for~$(\lambda_1, \dots, \lambda_n; 2 , \dots, 2)$. 
\end{mydef}

\section{Stratification of the Stacks \texorpdfstring{$\Mf_\dimvectd$}{Md}} \label{section3}
In this section we give stratification of the stack of torsion sheaves on a weighted projective line for weights \[w_1 = w_2 = \cdots = w_n = 2.\] 

\begin{mydef}
	We say that an algebraic stack~$\mathfrak{X}$ has an affine paving by~$\mathfrak{S}_i$, if~$\mathfrak{X}$ has a filtration
	\[
	\mathfrak{X} = \mathfrak{X}_N \supseteq \cdots \supseteq \mathfrak{X}_1 \supseteq \mathfrak{X}_0 = \emptyset
	\]
	by closed substacks such that every successive complement~$\mathfrak{S}_i = \mathfrak{X}_i \setminus \mathfrak{X}_{i-1}$ is isomorphic to a quotient stack of the form~$[\mathbb{A}^{n_i} / G_i]$ where~$G_i$ is a connected linear algebraic group with reductive part isomorphic to~$ \prod_j \GL_{d_{i , j}}(\C)$ for some~$d_{i , j} \in \N_0$.
\end{mydef}

Note that~$A_\bullet(\mathfrak{S}_i)= \bigotimes_j A_\bullet^{\GL_{d_{i , j}}(\C)}(\pt) [\text{shift}]$, the cycle map~$A_{\frac{1}{2}\bullet}(\mathfrak{S}_i) \to H^{\mathrm{BM}}_\bullet( \mathfrak{S}_i )$ is an isomorphism, and~$H_{2k}^\mathrm{BM}(\mathfrak{S}_i)$ is a pure Hodge structure of weight~$-2k$ and level~$0$, while the odd Borel--Moore homology groups~$H^\mathrm{BM}_{2k+1}(\mathfrak{S}_i) $ vanish.


\begin{prop}[{{\cite[Lemma~5.3 and Corollary~5.4]{FRChowHa}}}] \label{cylce map iso if iso on strata}
	For an algebraic stack~$\mathfrak{X}$ with an affine paving, the cycle map
	\[
	\mathrm{cl}_\mathfrak{X}\colon A_{\frac{1}{2} \bullet}(\mathfrak{X}) \to H_{\bullet}^{\mathrm{BM}}(\mathfrak{X}) 
	\]
	is an isomorphism and, moreover, we have
	\[
	A_\bullet(\mathfrak{X}) \cong \bigoplus_{i=0}^N A_\bullet(\mathfrak{S}_i).
	\]
\end{prop}

\begin{cor}\label{Purity of cohomology if affine paving}
	All odd Borel--Moore homology groups~$H_{2k+1}^\mathrm{BM}(\mathfrak{X})$ vanish and the even group~$H_{2k}^{\mathrm{BM}}(\mathfrak{X}) \cong \bigoplus_{i=0}^N H_{2k}^\mathrm{BM}(\mathfrak{S}_i)$ is a pure Hodge structure of weight~$-2k$ and level~$0$.
\end{cor}
\begin{proof}
	The isomorphism~$H_\bullet^{\mathrm{BM}}(\mathfrak{X}) \cong \bigoplus_{i=0}^N H_\bullet^{\mathrm{BM}}(\mathfrak{S}_i)$ is an isomorphism of mixed Hodge structures.
\end{proof}
\mycomment{\begin{la}[ref: la 5.3+cor5.4 in chowha paper von Hans und Markus] \label{cylce map iso if iso on strata alt} Suppose that a scheme~$X$ with a~$G$-action has a filtration~$X=X_N \supseteq \dots \supseteq X_0 \supseteq X_{-1} = \emptyset$ by closed~$G$-invariant subschemes such that the equivariant cycle maps for the successive complements~$S_n= X_n- X_{n-1}$ are isomorphisms for all~$n$.
		Then~$\mathrm{cl}_X^G$ is an isomorphism and, moreover, we have
		\[
		A_\bullet^G(X) \cong \bigoplus_{i=0}^N A_\bullet^G(S_i) \text{ and } A^\bullet_G(X)\cong \bigoplus_{i=0}^N A^{\bullet - \codim_X S_i}_G(S_i).
		\]
\end{la}}

\begin{cor}\label{KünnethIsoIfAffinePaving}
	If~$\mathfrak{X}$ and~$\mathfrak{Y}$ have affine pavings, the Künneth morphism 
	\[
	A_\bullet(\mathfrak{X}) \otimes_\Q A_\bullet(\mathfrak{Y}) \to A_\bullet(\mathfrak{X}\times \mathfrak{Y})
	\]
	is an isomorphism.
\end{cor}

\mycomment{
	\begin{la}\label{crtierion for balanced product}
		Let~$G$ be a connected algebraic group acting on a normal algebraic variety~$X$ such that there is a connected closed subvariety~$Y\subseteq X$ and a closed subgroup~$H\subseteq G$ such that~$G.Y=X$ and~$g.y\in Y$ implies~$g\in H$ for all~$y\in Y$.
		Then 
		\[
		X\cong G\times^H Y .
		\]
	\end{la}
	\begin{proof}
		The natural map~$G\times^H Y \to X, (g,y)\mapsto g.y$ is surjective because~$G.Y=X$ and it is injective because 
		\begin{align*}
			g.y=g'.y' & \implies g'^{-1}g.y\in Y \\
			& \implies g'^{-1}g \in H \\
			& \implies \overline{(g,y)} = \overline{(g(g^{-1}g'), (g'^{-1}g).y)}= \overline{(g',y')}.
		\end{align*}
		So, we have a bijective morphisms between connected varieties with normal codomain, which must be an isomorphism by ZMT as we are in characteristic zero.
	\end{proof}
}

To construct affine pavings, we need the following well-known fact about quotient stacks, which is essentially already in \cite{SerreEspacesFibresAlg}, but also see \cite[Lemma~7.2.3.2]{moduli}.

\begin{la}[Induction Isomorphism] \label{induction iso}
	Let~$X = G \times^H Y$ for two linear algebraic groups~$H \subseteq G$ with~$H$ acting on a variety~$Y$.
	Then~$[X / G] \cong [Y / H]$.
	\mycomment{
		If~$X\cong G\times^H Y$ then~$\mathbb{H}_G(X)\cong \mathbb{H}_H(Y)$ where~$\mathbb{H}_G$ denotes equivariant cohomology, equivariant Borel--Moore homology, or equivariant Chow groups/rings.
		These isomorphisms commute with the cycle map and with Poincaré duality.}
\end{la}

In this section, we show the following:

\begin{thm}\label{Affine stratification C(2^n)}
	Given a canonical algebra~$C(\boldsymbol{\lambda}; \mathbf{w})$ with~$w_i=2$ for~$i = 1, \dots , n$, for every regular dimension vector~$\dimvectd$, we have that the stack of regular representations of~$C(\boldsymbol{\lambda}; \mathbf{w})$ has an affine paving.
\end{thm}

\begin{cor}\label{A=H for Pn}
	The cycle map~$A_{\frac{1}{2}\bullet}(\Mf_\dimvectd)\to H_\bullet^{\mathrm{BM}}(\Mf_\dimvectd)$ is an isomorphism and thus we also have
	\[
	\Chowha(\Pbb^1(2^n)) = \Coha(\Pbb^1(2^n)).
	\]
\end{cor}

Before we prove \Cref{Affine stratification C(2^n)}, we recall the special case of~$n = 0$, that is the case where~$C(\boldsymbol{\lambda}; \mathbf{w}) = \C K_2$ is the path algebra of the~$2$-Kronecker quiver~$K_2 = \begin{tikzcd}
	\bullet \rar[bend left] \rar[bend right] & \bullet
\end{tikzcd}$.
We recall the stratification from \cite[Section~10.2]{FRChowHa}, which was constructed in order to calculate the Poincaré series of the semistable slope~$0$ CoHA of~$K_2$.

\begin{prop}\label{Affine stratification K_2}
	Given the~$2$-Kronecker quiver~$K_2$, for every regular dimension vector~$(d_0,d_0)$, the stack of regular representations has an affine paving.
\end{prop}
\begin{proof}
	Recall that all indecomposable representations of~$K_2$ are known:
	up to isomorphism, there is a unique indecomposable representation~$P_n$, resp.~$I_n$, for each dimension vector~$(n, n +1 )$, resp.~$(n+1 , n)$, for~$n\geq 0$, and there exist one-parametric families~$R_n(\lambda)$ of indecomposables for the dimension vectors~$(n,n)$ for~$n\geq 0$ and~$\lambda \in \Pbb^1(\C)$.
	The indecomposable representations of dimension vector~$(n,n)$ are exactly the regular indecomposables.
	Here, the representation~$R_n(\lambda)$ is given explicitly by the matrices~$(E_n, \lambda E_n + J_n)$ for~$\lambda \neq \infty$, and by~$(J_n, E_n)$ for~$\lambda =\infty$, where~$E_n$ denotes the~$n \times n$-identity matrix and~$J_n$ is the nilpotent~$n\times n$-Jordan block.
	
	Any semistable representation~$M$ of dimension~$(d_0, d_0)$ is therefore of the form
	\[
	M = R_{n_1}(\lambda_1) \oplus \cdots \oplus R_{n_k}(\lambda_k)
	\]
	for~$d_0 = n_1 + \cdots + n_k$ and~$\lambda_1, \dots, \lambda_k \in \Pbb^1(\C)$ uniquely defined up to reordering.
	We reorder the direct sum and assume~$\lambda_1, \dots, \lambda_j \neq \infty$,~$\lambda_{j+1} = \cdots = \lambda_k = \infty$, and~$n_{j+1}\geq n_{j+2} \geq \cdots \geq n_k$.
	We see that~$M$ is represented by a pair of block matrices
	\[
	\left( \twoByTwoMatrix{E_r}{0}{0}{J_\varpi}, \twoByTwoMatrix{A}{0}{0}{E_{d_0-r}}	\right)
	\]
	for some~$0 \leq r \leq d_0$ and some partition~$\varpi$ of~$d_0 - r$.
	We define~$S_{r; \varpi}$ as the~$G_{(d_0,d_0)}$-saturation of the set of pairs of such matrices and also set~$S_r\coloneqq \coprod_{\varpi \vdash (d_0 - r)} S_{r; \varpi}$.
	
	In \cite[Section~10.2]{FRChowHa}, it is shown that every~$S_r$ is locally closed, their union equals the semistable locus, the closure of~$S_r$ equals the union of~$S_{r'}$ for~$r'\leq r$, and 
	\[
	S_r \cong G_{(d_0 , d_0)} \times^{\GL_r(\C) \times \GL_{d_0 - r }(\C)} (M_r(\C) \times N_{d_0 - r}(C)),
	\]
	where~$N_{d_0 - r}(\C)$ denotes the set of all nilpotent~$(d_0 - r)\times (d_0 - r)$-matrices and the group~$\GL_r(\C)\times \GL_{d_0 - r}(\C)$ is considered as a subgroup of~$G_{(d_0,d_0)} = \GL_{d_0}(\C)^2$ by mapping a pair~$(g_1, h_4)$ to~$\left( \left(\begin{smallmatrix} g_1 & 0 \\ 0 & h_4 \end{smallmatrix}\right) ,  \left(\begin{smallmatrix} g_1 & 0 \\ 0 & h_4 \end{smallmatrix}\right) \right)$.
	
	\hyphenation{para-metrized}
	As the nilpotent cone~$N_{d_0 - r}(\C)$ is stratified by the~$\GL_{d_0 - r }(\C)$-orbits, which are parametrized by partitions of~$d_0 - r$, the variety~$S_r$ is stratified by the~$S_{r ; \varpi}$ and thus~$S_{r; \varpi}$ is locally closed, their union equals the semistable locus, the closure equals a union of smaller (in a suitable sense) such strata, and we have
	\begin{equation} \label{strata of R(K2) balanced product form}
		\begin{aligned}
			S_{r; \varpi} & \cong G_{(d_0,d_0)} \times^{\GL_r(\C) \times \GL_{d_0 - r }(\C)} (M_r(\C) \times (\GL_{d_0 - r }(\C) \cdot J_\varpi)) \\
			& \cong G_{(d_0,d_0)} \times^{\GL_r(\C) \times \GL_{d_0 - r }(\C)} (M_r(\C) \times \GL_{d_0 - r }(\C)/G_\varpi) \\
			& \cong G_{(d_0,d_0)} \times^{\GL_r(\C) \times G_\varpi} (M_r(\C) \times \{ J_\varpi\}),
		\end{aligned}
	\end{equation}
	where~$G_\varpi$ is the stabilizer of~$J_\varpi$ in~$\GL_{d_0 - r}(\C)$.
	\mycomment{These stabilizer groups have dimension~$\langle \varpi , \varpi \rangle = \sum_{i,j} \min(\varpi_i, \varpi_j)$,}
	The reductive part of~$G_\varpi$ is isomorphic to~$\prod_i \GL_{m_i}(\C)$, where~$m_i$ denotes the multiplicity of~$i$ as a part of~$\varpi$ for~$i \geq 1$.
	
	We therefore have for any dimension vector~$(d_0, d_0)$ an affine paving of the quotient stack~$[R_{{(d_0,d_0)}}^\sst(K_2) / G_{(d_0,d_0)} ]$ by 
	\[
	[ S_{r; \varpi} / G_{(d_0,d_0)} ] = [ (M_r(\C)\times \{ J_\varpi \}) / (\GL_r(\C)\times G_\varpi) ] \cong [ \mathbb{A}^{r^2} / (\GL_r(\C) \times G_\varpi) ],
	\]
	where the first equality follows from~\Cref{induction iso}.
\end{proof}

\begin{proof}[Proof of \Cref{Affine stratification C(2^n)}]
	
	We make the assumption~$\lambda_i\neq \infty$ and write in a slight abuse of notation~$\lambda_i = [\lambda_i : 1]$.
	Note that this assumption is fine, because of the action of~$\PGL_2(\C)$, see \Cref{PGL2 on C(lambdaw)}.
	(The case~$\lambda_i = \infty$ for some~$i$ could be considered in exactly the same way, but it would require an additional case distinction.)
	
	Consider the morphism
	\[
	R_\dimvectd\to R_{(d_0,d_0)}(\Pbb^1(2^{0})) = R^\sst_{(d_0, d_0)}(K_2), \quad (\alpha, \beta, A_1^{(1)}\!, \dots, A_n^{(2)}) \mapsto (\alpha, \beta)
	\]
	and define~$\widetilde{S}_{r ; \varpi}\coloneqq \widetilde{S}_{r ; \varpi}(\dimvectd)$ as the inverse image of~$S_{r ; \varpi}$ under this map.
	Obviously, we have an equivariant stratification~$R_\dimvectd=\coprod \widetilde{S}_{r ; \varpi}$ with locally closed~$G_\dimvectd$-invariant strata such that the closure of any one stratum is again a union of strata.
	It thus suffices to show that every~$\widetilde{S}_{r ; \varpi}$ has an affine paving.
	
	Note that for~$r = d_0$, there is only one partition and~$\widetilde{S}_{d_0}(\dimvectd)$ is open in~$R_\dimvectd$.
	We first reduce the case of a general~$\widetilde{S}_{r ; \varpi}$ to this special case.
	
	
	\newcommand{\vin}{\rotatebox[origin=c]{-90}{$\in$}}
	\textbf{Step 1:} We show that
	\begin{equation}\label{Sr tilde as balanced product}
		\widetilde{S}_{r ; \varpi}(\dimvectd)\cong G_\dimvectd\times^{\GL_r\times \GL_{d_1-(d_0-r)}\times \dots \times \GL_{d_n-(d_0-r)}\times G_\varpi}(\widetilde{M}_r(\dimvectd)\times \{J_\varpi\}),
	\end{equation}
	where
	\[
	\widetilde{M}_r(\dimvectd) \coloneqq \{ (\underset{\substack{\vin \\ \C^{r\times r}}}{A}, \underset{\substack{\vin \\ \C^{(d_1-d_0+r)\times r}}}{A_1^{(1)}} , \dots, \underset{\substack{\vin \\ \C^{r\times (d_n-d_0+r)}}}{ A_n^{ ( 2 ) } } ) \ \vert \ \lambda_i E_r + A = A_i^{ ( 2 ) } A_i^{ ( 1 ) } \! , \text{ for~$i=1,\dots, n$}\};
	\]
	this implies -- using \Cref{induction iso} -- that~$[\widetilde{S}_{r ; \varpi}(\dimvectd) / G_\dimvectd ] \cong [ \widetilde{S}_{r}(\dimvectd') / G_{\dimvectd'} ] \times [ \pt / G_\varpi]$ for
	\[
	\dimvectd'=(r,d_1-(d_0-r),\dots, d_n-(d_0-r), r),
	\]
	thereby reducing the whole problem to the open stratum, i.e., the special case of~$r = d_0$.
	
	By the construction of the stratification of regular representations of the Kronecker quiver, any element~$M = (\alpha,\beta, A_1^{ ( 1 ) } , \dots, A_n^{ ( 2 ) } ) \in \widetilde{S}_{r ; \varpi}$ is equivalent to one of the form 
	\[
	\left( \twoByTwoMatrix{E_r}{0}{0}{J_\varpi}, \twoByTwoMatrix{A}{0}{0}{E_{d_0-r}}, A_1^{(1)}, \dots, A_n^{(2)} \right)
	\]
	and there is an open subset on which the elements~$A$ and~$A^{(i)}_k$ are constructed by polynomials from~$M$, because the group~$\GL_r(\C)\times \GL_{d_0 - r}(\C)$ is special.
	
	Note that because~$\lambda_i \alpha +\beta$ factors through~$\C^{d_i}$, we have~$d_i \geq \rank(\lambda_i \alpha +\beta)\geq d_0-r$; i.e., for other choices of~$r$ and~$\dimvectd$ the scheme~$\widetilde{S}_r(\dimvectd)$ is empty.
	We now write
	\mycomment{\[
		A_i^{ ( 2 ) } = \left(\begin{smallmatrix} A_i^{(2)'} \\ A_i^{(2)''} \end{smallmatrix} \right)
		\]
		and~$A_i^{(1)} = \left(\begin{smallmatrix}
			A_i^{(1)'} & A_i^{ ( 1 )'' }
		\end{smallmatrix}\right)$ via~$d_0= r + (d_0-r)$.}
	$
	A_i^{(1)} = \left(\begin{smallmatrix}
		A_i^{(1),1} & A_i^{ ( 1 ),2 }
	\end{smallmatrix}\right)$ {and} $A_i^{ ( 2 ) } = \left(\begin{smallmatrix} A_i^{(2),1} \\ A_i^{(2),2} \end{smallmatrix} \right)
	$
	via~$d_0= r + (d_0-r)$.
	Because
	\[
	A_i^{ ( 2 ) } A_i^{ ( 1 ) } = \twoByTwoMatrix{ A + \lambda_i E_r  }{0}{0}{E_{d_0-r} + \lambda_i J_\varpi },
	\]
	we know that~$A_i^{(2),2} A_i^{(1),2} = E_{d_0-r}+\lambda_i J_\varpi$ is invertible.	
	We may therefore find matrices~$C_i \in \GL_{d_i}(\C)$ such that
	\mycomment{\[
		C_i A_i^{ ( 1 ) } = \twoByTwoMatrix{\widetilde{\widetilde{A_i^{ ( 1 ) }}}}{E_{d_0-r}}{\widetilde{A_i^{(1)}}}{0}.
		\]}
	\[
	C_i A_i^{ ( 1 ) } = \twoByTwoMatrix{ B_{i,2} }{ E_{d_0-r} }{ B_{i,1} }{ 0 }.
	\]
	The matrix~$C_i$ is also constructed in a polynomial way (on the open subset where a given~$(d_0 - r)\times (d_0 - r)$-minor of~$(E_{d_0 - r}  + \lambda_i J_\varpi)^{-1} A_i^{(2),2}$ does not vanish, it can be given explicitly).
	We replace~$A_i^{ ( 1 ) }$ by~$C_i A_i^{ ( 1 ) }$ and~$A_i^{ ( 2 ) }$ by~$A_i^{ ( 2 ) }C_i^{-1}$.
	
	We see that if we write
	\[
	A_i^{(2)} = \twoByTwoMatrix{ a_{ i, 1 } }{ a_{ i , 2 } }{ a_{ i , 3 } }{ a_{ i , 4 } },
	\]
	we obtain~$ a_{ i , 1 } = 0~$ and~$a_{ i , 3 }  = E_{d_0-r} + \lambda_i J_\varpi$ from the product of~$A_i^{ ( 2 ) }$ and~$A_i^{ ( 1 ) }$.
	We have
	\[
	A_i^{ ( 2 ) }= \twoByTwoMatrix{0}{ a_{ i , 2 } }{ E_{d_0-r} + \lambda_i J_\varpi }{ a_{ i , 4 } }, 
	\]
	which multiplied with
	\[
	D_i = \twoByTwoMatrix{ E_{d_0-r} }{ - (E_{d_0-r} + \lambda_i J_\varpi)^{-1} a_{ i , 4 } }{ 0 }{ E_{ d_i - ( d_0 - r ) }}
	\]
	from the right gives
	\[
	A_i^{(2)} D_i = \twoByTwoMatrix{0}{ \widehat{A}_i^{ ( 2 ) } }{E_{d_0-r} + \lambda_i J_\varpi}{0}.
	\]
	It follows that
	\[
	D_i^{-1} A_i^{(1)} = \twoByTwoMatrix{0}{E_{d_0-r}}{\widehat{A}_i^{ ( 1 ) }}{0}, 
	\]
	and so we replace~$A_i^{ ( 1 ) }$ by~$D_i^{-1} A_i^{(1)}$ and~$A_i^{ ( 2 ) }$ by~$ A_i^{ ( 2 ) } D_i$.
	
	We have now found a closed subset~$Y = \widetilde{M}_r(\dimvectd')\times \{ J_\varpi \}$ in~$\widetilde{S}_{r ; \varpi}$, such that~$G_\dimvectd Y = \widetilde{S}_{r ; \varpi}$, as well as an open cover of~$\widetilde{S}_{r ; \varpi} = \bigcup_i U_i$ together with morphisms
	\[
	(\varphi_1^{(i)}, \varphi_2^{(i)})\colon U_i \to G_\dimvectd \times Y
	\]
	such that~$\varphi_1^{(i)}(u) \cdot \varphi_2^{(i)}(u) = u$.
	
	Now consider the stabilizer of~$Y$.
	We take~$g_0,\, g_\infty \in \GL_{d_0}(\C)$ and~$g_i \in \GL_{d_i}(\C)$, written as block matrices
	\[
	g_0 = \twoByTwoMatrix{ g^{(1)} }{ g^{(2)} }{ g^{(3)} }{ g^{(4)} }\!, \qquad g_\infty = \twoByTwoMatrix{ g_\infty^{(1)} }{ g_\infty^{(2)} }{ g_\infty^{(3)} }{ g_\infty^{(4)} }\!,\qquad  g_i = \twoByTwoMatrix{ g^{(1)}_i }{ g^{ ( 2 ) }_i }{ g^{ ( 3 ) }_i }{ g^{ ( 4 )}_i}\! .
	\]
	Because~$(g_0, g_\infty)$ stabilizes elements of the form
	\[
	\left( \twoByTwoMatrix{E_r}{0}{0}{J_\varpi}, \twoByTwoMatrix{A}{0}{0}{E_{d_0-r}}	\right), 
	\]
	we have by~\eqref{strata of R(K2) balanced product form} that
	\[
	g_0 = g_\infty = \twoByTwoMatrix{ g_0^{(1)} }{0}{0}{ g_0^{(4)} }.
	\]
	Now we consider
	\begin{align*}
		\twoByTwoMatrix{ g^{(1)}_i }{ g^{(2)}_i }{ g^{(3)}_i }{g^{(4)}_i } \twoByTwoMatrix{0}{E_{d_0-r}}{a_i^{(1)}}{0} & = \twoByTwoMatrix{0}{E_{d_0-r}}{ b_i^{(1)}}{0} \twoByTwoMatrix{ g_0^{(1)} }{ 0 }{ 0 }{ g_0^{(4)} } \\
		\twoByTwoMatrix{ g_0^{(1)} }{ 0 }{ 0 }{ g_0^{(4)} } \twoByTwoMatrix{0}{ a_i^{(2)} }{ E_{d_0-r } + \lambda_i N }{ 0 } & = \twoByTwoMatrix{ 0 }{ b_i^{(2)} }{ E_{d_0-r} + \lambda_i N' }{ 0 } \twoByTwoMatrix{g^{(i)}_1}{g^{(i)}_2}{g^{(i)}_3}{g^{(i)}_4} \! .
	\end{align*}
	We deduce that~$g^{(2)}_i = 0$,~$g^{(3)}_i = 0$, and~$g^{(1)}_i = g_0^{(4)}$.
	We therefore have a bijective morphism
	\[
	G_\dimvectd\times^{\GL_r\times \GL_{d_1-(d_0-r)}\times \dots \times \GL_{d_n-(d_0-r)}\times \GL_{d_0-r}}(\widetilde{M}_r(\dimvectd) \times \{ J_\varpi \}) \to \widetilde{S}_{r ; \varpi}(\dimvectd).
	\]
	However, on the open subsets~$U_i\subseteq \widetilde{S}_{r ; \varpi}(\dimvectd)$ we have a polynomial inverse.
	The map is therefore an isomorphism.
	
	\textbf{Step 2:} Find an affine paving of~$\widetilde{S}_{d_0} = \widetilde{S}_{d_0}((d_0,d_1,\dots, d_n, d_0))$.
	
	For~$\underline{r}= (r_1,\dots , r_n)\in \N_0^n$ we denote by~$\widetilde{S}_{\underline{r}}$ the set of elements~$(\alpha, \beta, A_1^{(1)} \! , \dots , A_i^{(2)})$ in~$\widetilde{S}_{d_0}$ such that~$\rank(A_i^{(1)}) = r_i$.
	Obviously, this gives a finite~$G_\dimvectd$-invariant stratification of~$\widetilde{S}_{d_0}$ and we claim that~$[ \widetilde{S}_{\underline{r}} / G_\dimvectd ] \cong [\mathbb{A}^{N_{\underline{r} }} / G_{\underline{r} } ]$ for a group~$G_{\underline{r}}$ whose reductive part is a product of general linear groups.
	
	We write~$s_i\coloneqq  \dim \ker(A_i^{ (1) }) = d_0-r_i$ and~$ s = \sum_i s_i$.
	We also abbreviate~$s_{<i} \coloneqq \sum_{j=1}^{i-1} s_j$ and~$s_{>i}\coloneqq \sum_{j=i+1}^{n} s_j$.
	
	Choose a basis for each~$\ker(A_i^{(1)})$ in~$\C^{d_0}$.
	These are linearly independent, because an element of~$\ker( A_i^{ (1) } )$ is an eigenvector of~$A$ with eigenvalue~$-\lambda_i$, as~$A_i^{ (2) }A_i^{ (1) } = \lambda_i E_{d_0} + A$. 
	(Here we need the assumption~$\lambda_i \neq \infty$.)
	We can therefore find a basis of~$\C^{d_0}$ such that
	\mycomment{
		\[\NiceMatrixOptions{renew-dots,renew-matrix}
		A = \begin{pNiceArray}{ccccccc|c}
			-\lambda_1 &  			 &		 					&				 & 							& 				& 					   		&	\Block{3-1}{A^{(1)}}\\
			& \ddots & 							&				 & 							&				&  							& 	\phantom{A}\\
			& 			  & \!\! -\lambda_1	&				 & 							&				&  							&		\\
			& 			  &							&	\ddots	 &  						&				&  							&	\vdots \\
			& 			  & 						&				 & \!\! -\lambda_n   &				&  				   			&	\Block{3-1}{A^{(n)}}	\\
			& 			  & 						&				 &					  		&	\ddots	&  							& 					\\
			& 			  & 						&				 &				      		&				 & \!\! -\lambda_n 	&  				\\
			\hline
			\Block{1-7}{\mathbf{0}}
			&			&							&				&							&				&							& A'					
		\end{pNiceArray}.
		\]
	}
	
	\[\NiceMatrixOptions{renew-dots,renew-matrix}
	A = \begin{pNiceArray}{ccc|c}
		-\lambda_1 E_{s_1}  &  					& 					   		&	{A^{(1)}}\\
		&	\ddots	 &  					&	\vdots \\
		&				 & \!\! -\lambda_n E_{s_n}  &	{A^{(n)}}	\\
		\hline
		\Block{1-3}{\mathbf{0}}
		&				&							& A'					
	\end{pNiceArray}.
	\]
	We now also choose a basis of~$\C^{d_i}$ such that~$A_i^{ (1) }$ becomes
	\[
	A_i^{ (1) } = \begin{pmatrix}
		E_{s_{<i}} & 0 & 0 & 0 \\
		0 & 0 & E_{s_{>i}} & 0 \\
		0 & 0 & 0 & E_{d_0-s} \\
		0& 0 & 0& 0
	\end{pmatrix}
	\]
	as a~$(s_{<i}+s_{>i}+(d_0-s)+ (d_i-r_i))\times (s_{<i}+ s_i+s_{>i}+ (d_0-s))$-matrix.
	We also write
	\[
	A_i^{ (2) } = \begin{pmatrix}
		a^{(i)}_{11} & a^{(i)}_{12} & a^{(i)}_{13} & a^{(i)}_{14}  \\
		a^{(i)}_{21} & a^{(i)}_{22} & a^{(i)}_{23} & a^{(i)}_{24}  \\
		a^{(i)}_{31} & a^{(i)}_{32} & a^{(i)}_{33} & a^{(i)}_{34}  \\
		a^{(i)}_{41} & a^{(i)}_{42} & a^{(i)}_{43} & a^{(i)}_{44}  
	\end{pmatrix}
	\]
	as a~$(s_{<i}+ s_i + s_{>i} + (d_0-s))\times (s_{<i}+ s_{>i}+ (d_0-s) + (d_i-r_i))$-matrix and obtain from 
	\begin{align*}
		A_i^{ (2) } A_i^{(1)} & =
		\begin{pmatrix}
			a^{(i)}_{11} & 0 & a^{(i)}_{12} & a^{(i)}_{13}  \\
			a^{(i)}_{21} & 0 &  a^{(i)}_{22} & a^{(i)}_{23} \\
			a^{(i)}_{31} & 0 &  a^{(i)}_{32} & a^{(i)}_{33} \\
			a^{(i)}_{41} & 0 & a^{(i)}_{42} & a^{(i)}_{43} 
		\end{pmatrix} 
		= \lambda_i E_{d_0} + A
	\end{align*}
	
	\mycomment{ 
		\[
		\begin{pNiceArray}{ccccccccccc|c}
			\lambda_i-\lambda_1 &  			 &		 										&				 & 			& 				& 		&			&											& 				& 										& \Block{3-1}{A^{(1)}} \\
			& \Ddots & 												&				 & 			&				&  		& 			&											&				&										& 	\phantom{A}\\
			& 			  & \lambda_i -\lambda_1		&				 & 			 &				&  		&			&											& 				&										&	\\
			& 			  &												&	\Ddots	 & 			&				&  		&			&											&  				&										& 	\Vdots \\
			& 			  &												&				 & 0		&				&		& 			&											&				&										&	\Block{3-1}{A^{(i)}} \\
			& 			  &												&				 & 			&\Ddots 	&		& 			&											&				&										&	\phantom{A} \\
			& 			  &												&				 & 			&				&	0	& 			&											&				&										&	\\
			& 			  &												&				&			&				&		& \Ddots &											&				&										& \Vdots \\
			& 			 &  											&				 & 			&				&		&				& \lambda_i -\lambda_n   &				&  				   						&	\Block{3-1}{A^{(n)}}	\\
			& 			  & 											&				 &			&		  		&		&				&										  &	\Ddots	&  										& 					\\
			& 			  & 											&				 &			&	      		&		&				&	 									& 				& \lambda_i -\lambda_n 	&  				\\
			\hline																																																																														
			\Block{1-11}{\mathbf{0}}																																																																			
			&			&												&				&			&				&		&				&										&				&											& \lambda_i E_{d_0-s} + A'					
		\end{pNiceArray}
		\]
	}
	
	that 
	\begin{align*}
		a_{11}^{(i)} & = \begin{pNiceMatrix}
			(\lambda_i-\lambda_1) E_{s_{1}} \!\!\! & & \\
			& \Ddots & \\
			& & \!\!\!\!\!\!\!\!\!\!\!\! (\lambda_i-\lambda_{i-1}) E_{s_{i-1}}
		\end{pNiceMatrix}\mycomment{ \eqqcolon \lambda_i - \lambda_{<i}} ,\quad	a_{21}^{ (i) } = 0,\quad a_{31}^{ (i) } =  0,\quad a_{41}^{ (i) } = 0 \\
		a_{12}^{ (i) } & = 0,\quad a_{22}^{ (i) } = 0,\quad a_{32}^{ (i) } = \begin{pNiceMatrix}
			(\lambda_i-\lambda_{i+1}) E_{ s_{i+1} } \!\!\!\!\!\!\!\!\! & & \\
			& \Ddots & \\
			& & \!\!\!\!\!\!\!\!\!\!\!\! (\lambda_i-\lambda_{n}) E_{ s_n }
		\end{pNiceMatrix} \mycomment{\eqqcolon \lambda_i- \lambda_{>i}} ,\quad a_{42}^{ (i) } = 0, \\
		a_{13}^{ (i) } & = \begin{pNiceMatrix}
			A^{(1)} \\
			\Vdots \\
			A^{(i-1)}
		\end{pNiceMatrix}\!,\quad
		a_{23}^{ (i) } = A^{(i)}\!,\quad
		a_{33}^{ (i) } = \begin{pNiceMatrix}
			A^{(i+1)} \\
			\Vdots \\
			A^{(n)}
		\end{pNiceMatrix}\!,\quad a_{43}^{ (i) } = A'+\lambda_i E_{d_0-s}. 
	\end{align*}

	So, the first~$s_{<i}+s_{>i}+(d_0-s)=r_i$ columns of~$A_2^{(i)}$ are uniquely determined by~$A$ while we can freely choose the remaining~$( d_i - r_i )$ ones.
	
	Thus, we see that there is a closed subset~$Y_{\underline{r}} \cong \C^{d_0\times (d_0-s)}\times \prod_i \C^{d_0\times (d_i-r_i) }$ of~$\widetilde{S}_{\underline{r}}$ such that~$G_\dimvectd Y_{\underline{r}} = \widetilde{S}_{\underline{r}}$.
	We now need to find the stabilizer of the matrices in the above form.
	Given an element~$( g_0, g_1, \dots, g_n , g_\infty )\in G_\dimvectd$ fixing an element of this form, we see immediately that~$g_0 = g_\infty$, because it has to stabilize~$\alpha= E_{d_0}$.
	Because~$g_0$ has to fix~$\ker(A_1^{ (i) })$, we also have that
	\[
	g_0 = \begin{pNiceArray}{ccc|c}
		g^{ (1) }_0 & & & \widetilde{g}_0^{ (1) } \\
		& \Ddots & & \Vdots \\
		& & g^{(n)}_0 & \widetilde{g}_0^{ (n) } \\
		\hline
		\Block{1-3}{\mathbf{0}} & & &  g_0^{ (n+1) }
	\end{pNiceArray},
	\]
	where~$g_0^{ (i) } \in \GL_{s_i}(\C)$ for~$i=1 , \dots , n$.
	Next observe that~$A^{(i)}_1$ has to map~$\ker(A_1^{(j)})$ for~$j\neq i$ to the first basis vectors and the complement of~$\bigoplus_j \ker(A_1^{(j)})$ to the next~$d_0-s$ basis vectors.
	Therefore, we have 
	\[
	g_i = \begin{pNiceMatrix}
		g_0^{ (<i) } & 0 & \widetilde{g}_i^{ (<i) } & g^{(1)}_i \\
		0 & g_0^{ (>i) } & \widetilde{g}_i^{ (>i) } & g^{(2)}_i \\
		0 & 0 & g_0^{ (n+1) } & g^{ (3) }_i \\
		0 & 0 & 0 & g^{ (4) }_i 
	\end{pNiceMatrix},
	\]	
	with
	\[
	g_0^{ (<i) } = \begin{pNiceMatrix}
		g_0^{ (1) } & & \\
		& \Ddots & \\
		& & g_0^{ (i-1) }
	\end{pNiceMatrix} \qquad \text{and} \qquad
	g_0^{ (>i) }=\begin{pNiceMatrix}
		g_0^{(i+1)} & & \\
		& \Ddots & \\
		& & g_0^{ (n) }
	\end{pNiceMatrix}.
	\]
	We therefore see that there is a bijective morphism~$G_{\dimvectd} \times^{G_{\underline{r}}} Y_{\underline{r}} \to \widetilde{S}_{\underline{r}}$ where
	\[
	Y_{\underline{r}}  = \C^{d_0\times (d_0-s)} \times \prod_i \C^{d_0\times (d_i-r_i)} 
	\]
	and 
	\begin{align*}
		G_{\underline{r}} & = \left\{ (g_0, g_1, \dots , g_n, g_\infty ) \in G_{\dimvectd} \ \middle\vert \ 
		\begin{aligned}
			& g_0 = g_\infty = \begin{pNiceArray}{ccc|c}
				g_0^{(1)} & & & \widetilde{g}_0^{ (1) } \\
				& \Ddots & & \Vdots \\
				& & g_0^{ (n) } & \widetilde{g}_0^{ (n) } \\
				\hline
				\Block{1-3}{\mathbf{0}} & & &  g_{n+1}
			\end{pNiceArray} \\ 
			& \text{ and } 
			g_i = 
			\begin{pNiceMatrix}
				g_0^{ (<i) } & 0 & \widetilde{g}_0^{ (<i) } & g_i^{(1)} \\
				0 & g_0^{ (>i) } & \widetilde{g}_0^{ (>i) } & g^{(2)}_i \\
				0 & 0 & g_0^{ (n+1) } & g_i^{ (3) } \\
				0 & 0 & 0 & g_i^{ (4) } 
			\end{pNiceMatrix}
		\end{aligned}  \right\}.
	\end{align*}
	That this morphism is in fact an isomorphism can be shown in the same way as in Step~1, i.e., by locally constructing an inverse morphism.
	
	By \Cref{induction iso}, we thus have~$[ \widetilde{S}_{\underline{r}} / G_\dimvectd ] \cong [ Y_{\underline{r}} / G_{\underline{r}} ]$.
	Notice that the reductive part of~$G_{\underline{r}}$ is~$\prod_i \GL_{s_i}(\C) \times \GL_{d_0-s}(\C) \times \prod_i \GL_{d_i-r_i}(\C)$ and the unipotent part is (as a variety isomorphic to)~$\C^{s\times (d_0-s)} \times \C^{(d_0-s_i)\times (d_i-r_i)}$.
	\mycomment{
		Therefore, the equivariant cycle map is an isomorphism.

		From these two steps it follows that 
		\begin{align*}
			H_\bullet^{\mathrm{BM}, G_\dimvectd}(\widetilde{S}_r(d)) & = H_\bullet^{\mathrm{BM}, \GL_r\times \GL_{d_1-(d_0-r)}\times \dots \times \GL_{d_n-(d_0-r)}\times \GL_{d_0-r}}(\widetilde{M}_r \times N_{d_0-r}) \\
			& = H_\bullet^{\mathrm{BM}, \GL_r\times \GL_{d_1-(d_0-r)}\times \dots \times \GL_{d_n-(d_0-r)}}(\widetilde{M}_r(d)) \otimes H_\bullet^{\mathrm{BM}, \GL_{d_0-r}}(N_{d_0-r}) \\
			& = H_\bullet^{\mathrm{BM}, G_{d'}}(\widetilde{S}_{d'_0}(d')) \otimes H_\bullet^{\mathrm{BM}, \GL_{d_0-r}}(N_{d_0-r})\\
			(\text{Step 2}+) & = A_\bullet^{G_{d'}}(\widetilde{S}_{d'_0}(d')) \otimes A_\bullet^{ \GL_{d_0-r}}(N_{d_0-r}) \\
			& = A_\bullet^{\GL_r\times \GL_{d_1-(d_0-r)}\times \dots \times \GL_{d_n-(d_0-r)}}(\widetilde{M}_r(d)) \otimes A_\bullet^{\GL_{d_0-r}}(N_{d_0-r}) \\
			& = A_\bullet^{\GL_r\times \GL_{d_1-(d_0-r)}\times \dots \times \GL_{d_n-(d_0-r)}\times \GL_{d_0-r}}(\widetilde{M}_r \times N_{d_0-r}) \\
			& = A_\bullet^{G_\dimvectd}(\widetilde{S}_r(d)),
		\end{align*}
		where~$d'=(r, d_1-(d_0-r),\dots , d_n-(d_0-r), r)$.
	}
\end{proof}

\begin{rk}
	The statement \Cref{Affine stratification C(2^n)} is qualitative, but the proof allows for a quantitative description of~$A_\bullet(\Mf_\dimvectd)$ (using \Cref{cylce map iso if iso on strata}).
	It should therefore be possible to calculate the Poincaré series of~$\Coha(\Pbb^1(2^n))$ from this stratification.
	However, we will use a different approach in \Cref{Coha Poincare series}.
\end{rk}

\section{Functoriality of Hall Algebras}\label{section4}

In this section, we relate the Chow--Hall algebras of different weighted projective lines.
Let~$n\geq0$,~$(\boldsymbol{\lambda}; \mathbf{w}) \coloneqq (\lambda_1, \dots, \lambda_n; w_1,\dots, w_n)$, and~$(\widetilde{\boldsymbol{\lambda}}; \widetilde{\mathbf{w}}) \coloneqq (\lambda_1, \dots, \lambda_{n+1}; w_1, \dots, w_{n+1})$.
For a dimension vector~$\dimvectd$ of~$C(\boldsymbol{\lambda}; \mathbf{w})$ we write~$\Mf_\dimvectd(\boldsymbol{\lambda}; \mathbf{w})$ for the stack of regular representations of~$C(\boldsymbol{\lambda}; \mathbf{w})$ of dimension~$\dimvectd$, and for a dimension vector~$\dimvectd$ of~$C(\widetilde{\boldsymbol{\lambda}}; \widetilde{\mathbf{w}})$ we write~$\Mf_{\dimvectd}(\widetilde{\boldsymbol{\lambda}}; \widetilde{\mathbf{w}})$ for the stack of regular representations of~$C(\widetilde{\boldsymbol{\lambda}}; \widetilde{\mathbf{w}})$ of dimension~$\dimvectd$.

\subsection{Morphisms between Weighted Projective Lines}
By our definition of~$\Pbb^1(\boldsymbol{\lambda}; \mathbf{w})$, there is a morphism of stacks
\[
\Pbb^1(\widetilde{\boldsymbol{\lambda}}; \widetilde{\mathbf{w}}) = \Pbb^1(\lambda_1, \dots, \lambda_{n+1}; w_1, \dots, w_{n+1} ) \to \Pbb^1(\lambda_1, \dots, \lambda_{n}; w_1, \dots, w_{n}) = \Pbb^1(\boldsymbol{\lambda}; \mathbf{w})
\]
such that the composition
\[
\Pbb^1(\widetilde{\boldsymbol{\lambda}}; \widetilde{\mathbf{w}}) \to \Pbb^1(\boldsymbol{\lambda}; \mathbf{w}) \to \Pbb^1(\lambda_1,\dots, \lambda_{n-1}; w_1, \dots, w_{n-1}) \to \cdots \to \Pbb^1(\lambda_1; w_1) \to \Pbb^1
\]
coincides with the structure morphism~$\Pbb^1(\widetilde{\boldsymbol{\lambda}}; \widetilde{\mathbf{w}}) \to \Pbb^1$.
We consider the pushforward on torsion sheaves.
We see that the functor~$\Tor(\Pbb^1(\widetilde{\boldsymbol{\lambda}}; \widetilde{\mathbf{w}}) )  \to \Tor(\Pbb^1(\boldsymbol{\lambda}; \mathbf{w}))$ is exact and corresponds under the equivalence~$\Tor(\Pbb(\boldsymbol{\lambda}; \mathbf{w})) \simeq \Reg(C(\boldsymbol{\lambda}; \mathbf{w}))$ from \Cref{Tor(Pbb1) = Reg(C) stacky} to the functor~$\Rep(C(\widetilde{\boldsymbol{\lambda}}; \widetilde{\mathbf{w}})) \to \Rep(C(\boldsymbol{\lambda}; \mathbf{w}))$ given by
\begin{small}\[
	\begin{tikzcd}[column sep = tiny]
		& {\C^{d_{1_1}}} & \cdots & {\C^{d_{1_{w_1-1}}}} &&&&&& {\C^{d_{1_{1}}}} & \cdots & {\C^{d_{1_{w_1-1}}}} \\
		& \vdots & \vdots & \vdots & &&&& & \vdots & \vdots & \vdots \\
		{\C^{d_0}} &&&& {\C^{d_\infty}} &&&& {\C^{d_0}} &&&& {\C^{d_\infty}}. \\
		& {\C^{d_{n_1}}} & \cdots & {\C^{d_{n_{w_{n}-1}}}} &&&&&& {\C^{d_{n_{1}}}} & \cdots & {\C^{d_{n_{w_{n}-1}}}} \\
		& {\C^{d_{(n+1)_1}}} & \cdots & {\C^{d_{(n+1)_{w_{n+1}-1}}}}
		\arrow[from=1-2, to=1-3]
		\arrow[from=1-3, to=1-4]
		\arrow[from=1-4, to=3-5]
		\arrow[from=1-10, to=1-11]
		\arrow[from=1-11, to=1-12]
		\arrow[from=1-12, to=3-13]
		\arrow[from=3-1, to=1-2]
		\arrow[shift right, from=3-1, to=3-5]
		\arrow[shift left, from=3-1, to=3-5]
		\arrow[from=3-1, to=4-2]
		\arrow[from=3-1, to=5-2]
		\arrow[maps to, from=3-5, to=3-9]
		\arrow[from=3-9, to=1-10]
		\arrow[shift left, from=3-9, to=3-13]
		\arrow[shift right, from=3-9, to=3-13]
		\arrow[from=3-9, to=4-10]
		\arrow[from=4-2, to=4-3]
		\arrow[from=4-3, to=4-4]
		\arrow[from=4-4, to=3-5]
		\arrow[from=4-10, to=4-11]
		\arrow[from=4-11, to=4-12]
		\arrow[from=4-12, to=3-13]
		\arrow[from=5-2, to=5-3]
		\arrow[from=5-3, to=5-4]
		\arrow[from=5-4, to=3-5]
	\end{tikzcd}
\]\end{small}
This functor maps regular representations to regular representations by \Cref{rep of C sst iff Kronecker sst}.
It extends to a functor on the level of~$S$-valued points for any base scheme~$S$ and thus induces a morphism of stacks
\[
\Mf_{\dimvectd}(\widetilde{\boldsymbol{\lambda}}; \widetilde{\mathbf{w}})  \to \Mf_{F({\dimvectd})} (\boldsymbol{\lambda}; \mathbf{w}), 
\]
for every dimension vector~$\dimvectd$ of~$C(\widetilde{\boldsymbol{\lambda}}; \widetilde{\mathbf{w}})$, where
\[
F\! \left((d_0, d_{1_1}, \dots, d_{n_{ w_{n} - 1 }}, d_{(n+1)_1}, \dots, d_{(n+1)_{ w_{n+1} - 1 }} , d_0) \right) = (d_0, d_{1_1}, \dots, d_{n_{ w_{n} - 1 }}, d_0).
\]

Taking Chow rings (or cohomology), we obtain a map in the opposite direction
\[
A^\bullet(\Mf_{F({\dimvectd})}(\boldsymbol{\lambda}; \mathbf{w})) \to A^\bullet(\Mf_{{\dimvectd}}(\widetilde{\boldsymbol{\lambda}}; \widetilde{\mathbf{w}})).
\]
Given a dimension vector~$\dimvectd$ of~$C(\boldsymbol{\lambda}; \mathbf{w})$, let~$\tilde{\dimvectd}$ be the dimension vector of~$C(\widetilde{\boldsymbol{\lambda}}; \widetilde{\mathbf{w}})$ defined by~$F(\tilde{\dimvectd}) = \dimvectd$ and~$\tilde{\dimvectd}_{(n+1)_i}= d_0$ for~$i=1,\dots, w_{n+1}-1$.
This choice of inverse image of~$\dimvectd$ under~$F$ gives a morphism of graded vector spaces
\begin{equation}\label{FunctorialityMap}
	\Chowha(\Reg(C(\boldsymbol{\lambda}; \mathbf{w}))) = \bigoplus_\dimvectd A^\bullet (\Mf_\dimvectd(\boldsymbol{\lambda}; \mathbf{w})) \to \Chowha(\Reg(C(\widetilde{\boldsymbol{\lambda}}; \widetilde{\mathbf{w}})).
\end{equation}

\subsection{Functoriality of Hall Algebras}

We now want to show that the morphism~\eqref{FunctorialityMap} is not just a map of graded vector spaces but a homomorphism of algebras.
We need some preparation.

\begin{prop}\label{open substack of top stack isomorphic to bottom}
	For~$\dimvectd$ and~$\tilde{\dimvectd}$ are as above let~$\Mf^{\mathrm{inv}}_{\tilde{\dimvectd}}$ be the open substack of~$\Mf_{\tilde{\dimvectd}}$ of representations such that~$A_{n+1}^{(1)},\, A_{n+1}^{(2)}, \dots,\, A_{n+1}^{(w_{n+1}-1)}$ are invertible (that is all morphisms in the last arm except the last one).
	Then, the composition
	\[
	\Mf^{\mathrm{inv}}_{\tilde{\dimvectd}}\mono \Mf_{\tilde{\dimvectd} }\to \Mf_\dimvectd
	\]
	is an isomorphism.
	Similarly, the open substack~$\Mf_{\tilde{\dimvectd},\tilde{\dimvecte}}^{\mathrm{inv}}\subseteq \Mf_{\tilde{\dimvectd},\tilde{\dimvecte}}$ defined as the inverse image of~$\Mf^{\mathrm{inv}}_{\tilde{\dimvectd}+\tilde{\dimvecte}}$ under~$\tilde{p}\colon \Mf_{\tilde{\dimvectd},\tilde{\dimvecte}}\to \Mf_{\tilde{\dimvectd}+\tilde{\dimvecte}}$ satisfies that the composition
	\[
	\Mf^{\mathrm{inv}}_{\tilde{\dimvectd},\tilde{\dimvecte}}\mono \Mf_{\tilde{\dimvectd},\tilde{\dimvecte}} \to \Mf_{\dimvectd,\dimvecte}
	\]
	is an isomorphism.
\end{prop}
\begin{proof}
	We have~$\Mf_{\tilde{\dimvectd}}=[R_{\tilde{\dimvectd}}/G_{\tilde{\dimvectd}}]$ and~$\Mf_{\tilde{\dimvectd}}^{\mathrm{inv}}=[R_{\tilde{\dimvectd}}^{\mathrm{inv}}/G_{\tilde{\dimvectd}}]$, where
	\begin{small}\begin{align*}
		R_{\tilde{\dimvectd}}^{\mathrm{inv}} & = 
			\left\{(\alpha, \beta, A_1^{(1)}\!, \dots, A_{n}^{(w_{n})}\!, A_{n+1}^{(1)}, \dots , A_{n+1}^{ ( w_{n+1} ) }  ) \ \middle| \ \begin{aligned}
				& \lambda_i^{(0)} \alpha + \lambda_i^{(1)} \beta = A_i^{(w_i)}\cdots A_i^{(1)}  \ \text{and}  \\
				& A_{n+1}^{(1)}, \dots, A_{n+1}^{(w_{n+1}-1)} \ \text{are invertible}
			\end{aligned} \right\} 
		\\
		& = 
			\left\{(\alpha, \beta, A_1^{ (1) }\!, \dots,  A_{ n + 1 }^{ ( w_{n+1} ) }) \ \middle| \ \begin{aligned}
				& (\alpha, \dots , A_{n}^{ ( w_{n} ) } ) \in R_\dimvectd, \, A_{ n + 1 }^{(1)}, \dots, A_{ n + 1 }^{ (w_{n+1} - 1 ) } \in \GL_{d_0}(\C), \\
				& A_{ n + 1 }^{( w_{n+1} )}= (\lambda_{ n+1 }^{ ( 0 ) } \alpha + \lambda_{ n+1 }^{ ( 1 ) }\beta)(A_{n+1}^{ (1) })^{-1} \cdots (A_{n+1}^{ ( w_{n+1} - 1 ) })^{-1} 
			\end{aligned}\right\}  \\
		& = R_\dimvectd \times \GL_{d_0}(\C)^{ w_{n+1} -1 }
	\end{align*}
\end{small}
	\mycomment{
		\begin{align*}
			R_{\tilde{\dimvectd}}^{\mathrm{inv}} & = \left\{(\alpha, \beta, A_1^{(1)}, \dots, A_{n}^{(w_n)}, A_{n+1}^{(1)}, \dots , A_{n+1}^{ ( w_{n+1} ) }  ) \ \middle| \ \begin{aligned}
				& \lambda_i^{(0)} \alpha + \lambda_i^{(1)} \beta = A_i^{(w_i)}\cdots A_i^{(1)}  \text{ and } \\
				& A_{n+1}^{(1)}, \dots, A_{n+1}^{(w_{n+1}-1)} \text{ are invertible }
			\end{aligned} \right\} \\
			& = \left\{(\alpha, \beta, A_1^{ (1) }, \dots,  A_{ n + 1 }^{ ( w_{n+1} ) }) \ \middle| \ \begin{aligned}
				& (\alpha, \dots , A_{n}^{ ( w_n ) } ) \in R_\dimvectd, \, A_{ n+1 }^{(1)}, \dots, A_{ n+1 }^{ (w_{n+1} - 1 ) } \in \GL_{d_0}(\C), \\
				& A_{ n+1 }^{( w_{n+1} )}= (\lambda_{ n+1 }^{ ( 0 ) } \alpha + \lambda_{ n+1 }^{ ( 1 ) }\beta)(A_{n+1}^{ (1) })^{-1} \cdots (A_{n+1}^{ ( w_{n+1} - 1 ) })^{-1} 
			\end{aligned}\right\} \\
			& = R_\dimvectd \times \GL_{d_0}(\C)^{ w_{n+1} -1 }
	\end{align*}}
	and so we obtain
	\[
	\Mf_{\tilde{\dimvectd}}^{\mathrm{inv}}= [(R_\dimvectd\times \GL_{d_0}(\C)^{ w_{n+1} -1 })/(G_\dimvectd\times \GL_{d_0}(\C)^{ w_{n+1} - 1 } )] \cong [R_\dimvectd/G_\dimvectd]= \Mf_\dimvectd.
	\]
	A similar computation can be done for the stack of extensions.
\end{proof}


\begin{la}\label{functoriality morphisms are lci}
	Given~$\dimvectd$ and~$\tilde{\dimvectd}$ as above, then the morphisms~$R_{\tilde{\dimvectd}}\to R_\dimvectd$ and~$R_{\tilde{\dimvectd},\tilde{\dimvecte}}\to R_{\dimvectd , \dimvecte}$ are l.c.i. i.e., they factor as a regular closed immersion followed by a smooth morphism, see {{\cite[Section~6.6]{FultonIntersectionTheory}}}.  
\end{la}
\begin{proof}
	The first morphism factors as a closed immersion followed by a smooth morphism:
	\begin{equation*}
		\begin{tikzcd}[column sep = normal, ampersand replacement= \&, row sep = 0]
			{R_{\tilde{\dimvectd}}}\arrow[hook, r]  \&  {\left\{(\alpha, \beta, A_1^{(1)}\!, \dots , A_{n+1}^{(w_{n+1})}) \ \middle\vert \ \begin{aligned}
					& \lambda_i^{ ( 0 ) } \alpha+ \lambda_i^{ ( 1 ) } \beta = A_i^{(2)}A_i^{(1)}\!, \\
					& \text{for~$i=1,\dots, n$}
				\end{aligned} \right\}} = {{R_\dimvectd} \subnode{u1}{{}\times{}}  (\C^{d_0\times d_{0}})^{w_{n+1}} }  \\ 
			\& \phantom{{\left\{(\alpha, \beta, A_1^{(1)}\!, \dots , A_{n+1}^{(w_{n+1})}) \ \middle\vert \ \begin{aligned}
						& \lambda_i^{ ( 0 ) } \alpha+ \lambda_i^{ ( 1 ) } \beta = A_i^{(2)}A_i^{(1)}\!, 
						\\				& \text{for~$i=1,\dots, n$}
					\end{aligned} \right\}} = {}\times{} }\subnode{u2}{{}R_\dimvectd{}} \phantom{(\C^{d_0\times d_{0}})^{w_{n+1}} }
		\end{tikzcd}
		\begin{tikzpicture}[overlay, remember picture]
			\draw [->>] (u1) edge  (u2.north -| u1);
		\end{tikzpicture}
	\end{equation*}
	By \Cref{Rd smooth + dim} we see that the closed immersion is regular.
	The same argument with \Cref{Rde smooth + dim} shows the statement for~$R_{\dimvectd , \dimvecte}$.
\end{proof}

\mycomment{
	\begin{la}
		Let~$p\colon X\to Y$ be a proper morphism of relative dimension~$0$ with~$Y$ connected.
		Assume there is an open subset~$U\subseteq Y$ such that~$p^{-1}(U)\to U$ is an isomorphism.
		Then~$p_*(1_{H^\bullet(X)})= 1_{H^\bullet(Y)}$.
	\end{la}
	\begin{proof}
		The horizontal morphisms of the cartesian diagram
		\[
		\begin{tikzcd}
			p^{-1}(U) \rar{\sim} \dar{i} & U \dar{j} \\
			X \rar{p} & Y.
		\end{tikzcd}
		\]
		are proper of relative dimension~$0$.
		It follows that the diagram
		\[
		\begin{tikzcd}
			H^0(p^{-1}(U)) \rar{\sim}  & H^0(U)  \\
			H^0(X) \rar{p_*} \uar{i^*} & H^0(Y) \uar{j^*}
		\end{tikzcd}
		\]
		commutes.
		Therefore we have
		\begin{align*}
			j^*(p_*(1_{H^\bullet(X)}))= p_*(i^*(1_{H^\bullet(X)}))= p_*(1_{H^\bullet(p^{-1}(U))}) = 1_{H^\bullet(U)} = j^*(1_{H^\bullet(Y)}).
		\end{align*}
		The assumption that~$Y$ is connected implies that~$j^*$ is injective on zeroth cohomology.
	\end{proof}
}

For the proof of the next lemma, we need the notion of refined intersection products defined by Fulton in \cite[Section~8.1]{FultonIntersectionTheory}.
Associated to a diagram
\[
\begin{tikzcd}
	X' \dar{p_X} & Y' \dar{p_Y} \\
	X \rar{f} & Y
\end{tikzcd}
\]
with~$Y$ smooth of dimension~$n$, the refined intersection product is a map of the form
\[
-\cdot_f - \colon A_k(X') \otimes A_l(Y') \to A_{k+l-n}(X' \times_Y Y').
\]

\begin{la}\label{projection formula for birational morphism}
	Given two morphisms~$X\xrightarrow{f} Y \xrightarrow{g} Z$ between irreducible varieties such that~$g$ and~$g\circ f$ are l.c.i.,~$Z$ is smooth, and~$f$ is proper and birational. 
	Then, in Chow groups, the following diagram commutes:
	\[
	\begin{tikzcd}[column sep = tiny]
		{A_\bullet(X)} && {A_\bullet(Y)} \\
		& {A_\bullet(Z).}
		\arrow["{f_*}", from=1-1, to=1-3]
		\arrow["{(g\circ f)^*}", from=2-2, to=1-1]
		\arrow["{g^*}"', from=2-2, to=1-3]
	\end{tikzcd}
	\]
\end{la}
\begin{proof}
	Being birational and proper,~$f$ is surjective.
	Therefore we have~$f_*([X])=[Y]$.
	Applying the projection formula \cite[Proposition~8.1.1.(c)]{FultonIntersectionTheory} with~$p_X=\id_X$,~$p_Y= \id_Y$, and~$p_Z=\id_Z$, we have
	\[
	f_*(x\cdot_{g\circ f} z) = f_*(x)\cdot_g z.
	\]
	As we also have by compatibility of the refined intersection product with l.c.i\ pullbacks \cite[\mycomment{Definition~8.1.2 and }Proposition~8.1.2.(b)]{FultonIntersectionTheory} that~$g^*(z)= [Y]\cdot_g z$ and~$(g\circ f)^*(z)= [X]\cdot_{(g\circ f)} z$, it follows that
	\[
	f_*((g\circ f)^*(z))  = f_*([X]\cdot_{g\circ f} z) 
	= f_*([X])\cdot_g z 
	= [Y] \cdot_g z 
	= g^*(z). \qedhere
	\]
\end{proof}

Before we come to our main theorem of this section, we need the following trivial lemma from linear algebra:

\begin{la} \label{ProdOfSqMatsIsUpperTriangular}
	Given two square matrices~$A$,~$B\in \C^{n\times n}$ such that the product~$AB$ has a block from
	\[
	A  B = \twoByTwoMatrix{C_1}{C_2}{0}{C_3}
	\]
	with~$C_1  \in \C^{m\times m}$, then there is an invertible matrix~$g\in \GL_n(\C)$ such that 
	\[
	Ag = \twoByTwoMatrix{A_1}{A_2}{0}{A_3} \quad \text{and} \quad g^{-1}B = \twoByTwoMatrix{B_1}{B_2}{0}{B_3}
	\]
	with~$A_1$,~$B_1\in \C^{m\times m}$.
\end{la}


\begin{thm}\label{thm functorial}
	The maps~$\Chowha(\Pbb^1(\boldsymbol{\lambda}; \mathbf{w}))\to \Chowha(\Pbb^1(\widetilde{\boldsymbol{\lambda}}; \widetilde{\mathbf{w}}))$ are morphisms of algebras.
\end{thm}

\begin{proof}
	Let~$\dimvectd$ and~$\tilde{\dimvectd}$ be as above.
	The morphisms of stacks~$\Mf_{\tilde{\dimvectd}}\xrightarrow{{f}} \Mf_\dimvectd$ and~$\Mf_{\tilde{\dimvectd},\tilde{\dimvecte}}\xrightarrow{\tilde{f}} \Mf_{\dimvectd , \dimvecte}$ fit into the commutative diagram
	\[\begin{tikzcd}
		\Mf_{\tilde{\dimvectd} }\times \Mf_{\tilde{\dimvecte}} \dar & \Mf_{\tilde{\dimvectd},\tilde{\dimvecte}} \lar \rar{\tilde{p}} \dar{\tilde{f}} & \Mf_{\tilde{\dimvectd}+\tilde{\dimvecte}}\dar{f} \\
		\Mf_\dimvectd \times \Mf_\dimvecte & \Mf_{\dimvectd , \dimvecte} \lar \rar{p} & \Mf_{\dimvectd + \dimvecte}.
	\end{tikzcd}\]
	It suffices to show that
	\begin{equation}\label{functoriality goal}
		\begin{tikzcd}
			A_\bullet(\Mf_{\tilde{\dimvectd},\tilde{\dimvecte}})  \rar{\tilde{p}_*} & A_\bullet(\Mf_{\tilde{\dimvectd}+\tilde{\dimvecte}}) \\
			A_\bullet(\Mf_{\dimvectd , \dimvecte}) \rar{p_*} \uar{\tilde{f}^*} & A_\bullet(\Mf_{\dimvectd + \dimvecte}) \uar{f^*}
		\end{tikzcd}
	\end{equation}
	commutes.
	Consider the fiber product~$\mathfrak{X}\coloneqq \Mf_{\tilde{\dimvectd}+\tilde{\dimvecte}} \times_{\Mf_{\dimvectd + \dimvecte}} \Mf_{\dimvectd , \dimvecte}$ in the diagram
	\begin{equation}\label{pullback functoriality}\begin{tikzcd}
			\Mf_{\tilde{\dimvectd},\tilde{\dimvecte}}
			\arrow[bend left]{drr}{\tilde{p}}
			\arrow[bend right,swap]{ddr}{\tilde{f}}
			\arrow[dashed]{dr}[description]{h} & & \\
			& \mathfrak{X} \arrow{r}{p'} \arrow{d}[swap]{f'}
			& \Mf_{\tilde{\dimvectd}+\tilde{\dimvecte}} \arrow{d}{f} \\
			&  \Mf_{\dimvectd , \dimvecte} \arrow[swap]{r}{p}
			& \Mf_{\dimvectd + \dimvecte}
	\end{tikzcd}\end{equation}
	We get~$\mathfrak{X}=[R_{\tilde{\dimvectd}+\tilde{\dimvecte}}\times_{R_{\dimvectd + \dimvecte}} R_{\dimvectd , \dimvecte} / G_{\tilde{\dimvectd}+\tilde{\dimvecte}}\times_{G_{\dimvectd + \dimvecte}} G_{\dimvectd , \dimvecte}]\eqqcolon [ X / G ]$.
	The map~$p'$ is proper as the pullback of the proper~$p$, but then also~$h$ is proper by \cite[Tag:~0CPT]{stacks-project}.
	
	Observe that~$R_{\tilde{\dimvectd}, \tilde{\dimvecte}} \times^{G_{\tilde{\dimvectd},\tilde{\dimvecte}}} G \to  X$ is surjective by \Cref{ProdOfSqMatsIsUpperTriangular}.
	It follows that~$h$ is surjective and that~$X$ is irreducible.
	Additionally,~$h$ is birational by \Cref{open substack of top stack isomorphic to bottom}.
	
	Now, consider the induced diagram in Chow groups:
	\begin{equation}\label{functoriality goal with fiber product}
		\begin{tikzcd}
			A_\bullet(\Mf_{\tilde{\dimvectd},\tilde{\dimvecte}})
			\arrow[bend left]{drr}{\tilde{p}_*}
			\arrow{dr}[description]{h_*} & & \\
			& A_\bullet(\mathfrak{X}) \arrow{r}{(p')_*}  
			& A_\bullet(\Mf_{\tilde{\dimvectd}+\tilde{\dimvecte}}) \\
			&  A_\bullet(\Mf_{\dimvectd , \dimvecte}) \arrow[swap]{r}{p_*}\arrow[swap]{u}{f^!}\arrow[bend left]{uul}{\tilde{f}^*}
			& A_\bullet(\Mf_{\dimvectd + \dimvecte}) \arrow[swap]{u}{f^*},
		\end{tikzcd}
	\end{equation}
	where the morphism~$f^!$ is the refined Gysin pullback as defined in \cite[Chapter~6.6]{FultonIntersectionTheory}.
	The upper triangle commutes by functoriality and the lower square commutes by \cite[Theorem~6.2~(a) and Proposition~6.6]{FultonIntersectionTheory}.
	Now, we claim that~$f'$ is l.c.i.\ and~$f'^*=f^!$.
	Indeed, the fiber diagram decomposes into two fiber diagrams
	\[\begin{tikzcd}
		{\mathfrak{X}} & {[R_{\tilde{\dimvectd}+\tilde{\dimvecte}}/G_{\tilde{\dimvectd}+\tilde{\dimvecte}}]} \\
		{\widetilde{\mathfrak{X}}} & {[\widetilde{R_{\tilde{\dimvectd}+\tilde{\dimvecte}}}/G_{\tilde{\dimvectd}+\tilde{\dimvecte}}]} \\
		{[R_{\dimvectd , \dimvecte}/G_{\dimvectd , \dimvecte}]} & {[R_{\dimvectd + \dimvecte}/G_{\dimvectd + \dimvecte}]}.
		\arrow[from=1-1, to=1-2]
		\arrow["{f_1'}"', hook, from=1-1, to=2-1]
		\arrow["{f_1}", hook, from=1-2, to=2-2]
		\arrow[from=2-1, to=2-2]
		\arrow["{f_2'}"', from=2-1, to=3-1]
		\arrow["{f_2}", from=2-2, to=3-2]
		\arrow[from=3-1, to=3-2]
	\end{tikzcd}\]
	Here~$f=f_2\circ f_1$ is the factorization from \Cref{functoriality morphisms are lci}, i.e.,
	\[
	[\widetilde{R_{\tilde{\dimvectd}+\tilde{\dimvecte}}}/G_{\tilde{\dimvectd}+\tilde{\dimvecte}} ] = [ (R_{\dimvectd + \dimvecte}\times (\C^{d_0\times d_0})^{w_{n+1}}) / G_{\tilde{\dimvectd}+\tilde{\dimvecte}} ] .
	\]
	We have~$\widetilde{\mathfrak{X}} =  [ (R_{\dimvectd , \dimvecte}\times (\C^{d_0\times d_0})^{w_{n+1}} ) / G ]$.
	Because~$R_{\tilde{\dimvectd},\tilde{\dimvecte}} \mono R_{ \tilde{\dimvectd}, \tilde{\dimvecte} }^\sst(Q)$ is a regular embedding by \Cref{Rde smooth + dim},~$f_1'$ must also be a regular embedding and its codimension agrees with the codimension of~$f_1$.
	We deduce~$f^! \overset{\text{def.}}{=} f_1^! \circ f_2'^* = f_1'^*\circ f_2'^* \overset{\text{def.}}{=} f'^*$, using \cite[Remark~6.2.1]{FultonIntersectionTheory}.
	But then the left triangle in diagram~\eqref{functoriality goal with fiber product} commutes by \Cref{projection formula for birational morphism}.
	Thus, also diagram~\eqref{functoriality goal} commutes.
\end{proof}


\section{The Algebras \texorpdfstring{$\Pp_n$}{Pn}}\label{Section alg Ppn}\label{section5}

We now give an algebraic description of a new class~$\Pp_n$ of algebras in terms of generators and relations.
We will prove in \Cref{section7} that the algebra~$\Coha(\Pbb^1(2^n))$ is given by~$\Pp_n$.

For~$n \in \N_0$ ,let~$\Lambda_n^+$ be the monoid~$\Lambda_n^+\coloneqq \{ \dimvectd\in \N_0^{Q_0(2^n)} \ \vert \ d_0 = d_\infty \}$ and consider the elements
\begin{align*}
	\dnull & = (1,1,\dots, 1,1),\\
	\mathrm{e}_k & = (0,0,\dots, 1, \dots, 0, 0)\qquad \text{with~$1$ only at the~$k$-th place, and} \\
	\mathrm{f}_k &  = \dnull - \mathrm{e}_k = (1,1,\dots, 0, \dots, 1, 1) \qquad \text{with~$0$ only at the~$k$-th place}
\end{align*}
for~$k=1,\dots, n$ in~$\Lambda_n^+$.

We write~$[- , -]$ for the commutator and~$\{- , -\}$ for the anticommutator.

We define the algebra~$(\Pp_n, *)$ as the~$\Z\times \Lambda_n^+$-graded~$\Q$-algebra generated by the elements
\begin{align*}
	& e_{1, 1}, e_{1, 3}, e_{1,5}, \dots , e_{1, 2i+1}, \dots , & 	& f_{1, 1}, f_{1, 3}, f_{1, 5}, \dots , f_{1, 2i+1}, \dots , \\
	& e_{2, 1}, e_{2, 3}, e_{2, 5}, \dots , e_{2, 2i+1}, \dots , &  	& f_{2, 1}, f_{2, 3}, f_{2, 5},  \dots , \\
	& \quad \vdots & 	& \quad \vdots \\
	& e_{n, 1}, e_{n ,3}, \dots , & 	& f_{n, 1}, f_{n ,3}, \dots , \\
	& h_2, h_4, h_6, \dots, h_{2i+2}, \dots , & 	& g_0, g_2, g_4, \dots ,  g_{2i} , \dots 
\end{align*}
with degrees
\begin{gather*}
	\deg(e_{k, 2i+1})  = (2i+1, \mathrm{e}_k),\qquad  	\deg(f_{k ,2i+1}) = (2i+1, \mathrm{f}_k), \\
	\deg(h_{2i+2})  = (2i+2, \dnull) ,  \qquad \deg(g_{2i})  = (2i, \dnull ) 
\end{gather*}
subject to the relations
\begin{align}
	\{e_{k, 2i+1},e_{k, 2j+1}\} & =0, \label{Pn rel: ek and ek}\\
	\{f_{k, 2i+1},f_{k, 2j+1}\} & =0, \label{Pn rel: fk and fk}\\
	\{e_{k, 2i+1},f_{k, 2j+1}\} & = h_{2(i+j)+2},\label{Pn rel: ek and fk}\\
	[e_{k, 2i+1}, e_{l, 2j+1}] &  = [f_{k, 2i+1}, f_{l, 2j+1}] = [e_{k, 2i+1}, f_{l, 2j+1}] = 0 \qquad \text{for~$k\neq l$}, \label{Pn rel: commutation} \\
	[h_{2i+2}, e_{k, 2j+1}] & = [h_{2i+2}, f_{k, 2j+1}] = [h_{2i+2}, h_{2j+2}] = 0, \label{Pn rel: h commutes} \\
	[g_{2i}, e_{k, 2j+1}] & = \sum_{r=0}^{j-1} h_{2(i+j-r)}*e_{k, 2r+1}, \label{Pn rel: g and e commutator}\\
	\begin{split}
		[g_{2i},f_{k, 2j+1}] & = -[g_{2j},f_{k, 2i+1}], \label{Pn rel: g and f commutator}\\
		[g_{2i},f_{k, 2j+1}] & = \sum_{r=0}^{j-i-1} f_{k, 2(j-1-r)+1} * h_{2(r+i)+2} \qquad \text{for~$i\leq $}j, 		
	\end{split}\\
	\begin{split}
		[g_{2i},h_{2j+2}] & = -[g_{2j},h_{2i+2}], \label{Pn rel: g and h commutator}\\
		[g_{2i},h_{2j+2}] & = \sum_{r=0}^{j-i-1} h_{2(j-1-r)+2}* h_{2(r + i)+2} \qquad \text{for~$i\leq j$}, 		
	\end{split}\\
	[g_{2i} , g_{2j}] & = 2 \cdot \sum_{r=0}^{j-i-1} g_{2(j-1-r)}* h_{2(r + i) + 2} \qquad \text{for~$i\leq j$}. 		
\end{align}

In terms of the generating series
\begin{align*}
	E_k(X) & = \sum_{i\geq 0} e_{k, 2i+1} X^i & F_k(X) &= \sum_{i\geq 0} f_{k, 2i+1} X^i, \\
	G(X) & = \sum_{i\geq 0} g_{2i} X^i, & H(X) &= \sum_{i\geq 0} h_{2i+2} X^i ,
\end{align*}
where~$X$ and~$Y$ are formal variable that commute with all generators and with each other, these relations can be rewritten as
\begin{align}
	\{E_k(X),E_k(Y)\} & = 0, \label{Pn rel: Ek and Ek}\\
	\{F_k(X),F_k(Y)\} & = 0, \label{Pn rel: Fk and Fk}\\
	\{E_k(X),F_k(Y)\} & =\frac{YH(Y)-XH(X)}{Y-X},  \label{Pn rel: Ek and Fk}\\
	[E_k(X),E_l(Y)] & = 0  \qquad \text{for~$k \neq l$}, \label{Pn rel: Ek and El} \\
	[F_k(X),F_l(Y)] & = 0 \qquad \text{for~$k \neq l$}, \label{Pn rel: Fk and Fl}\\
	[E_k(X),F_l(Y)] & = 0 \qquad \text{for~$k \neq l$}, \label{Pn rel: Ek and Fl}\\
	[H(X),E_k(Y)] & = 0, \label{Pn rel: H and E}\\
	[H(X),F_k(Y)] & = 0, \label{Pn rel: H and F}\\
	[H(X),H(Y)] & = 0, \label{Pn rel: H and H}\\
	[G(X),E_k(Y)] & = YE_k(Y)*\frac{YH(Y)-XH(X)}{Y-X}, \label{Pn rel: G and E}\\
	[G(X),F_k(Y)] & = \frac{(YF_k(Y)-XF_k(X))*(YH(Y)-XH(X))}{Y-X}, \label{Pn rel: G and F}\\
	[G(X),H(Y)] & = \frac{(YH(Y)-XH(X))*(YH(Y)-XH(X))}{Y-X}, \label{Pn rel: G and H}\\
	[G(X),G(Y)] & = 2\frac{(YG(Y)-XG(X))*(YH(Y)-XH(X))}{Y-X} \label{Pn rel: G and G} .
\end{align}

\begin{rk}\label{PpnToPpn+1}
	We have a homomorphism of graded algebras~$\Pp_n \to \Pp_{n+1}$, which maps a generator of~$\Pp_n$ to the corresponding generator of the same name in~$\Pp_{n+1}$.
	This map is even injective as the next proposition shows.
\end{rk}

\begin{prop}\label{Ppn Poincare series}
	The algebra~$\Pp_n$ has a PBW-type basis given by ordered products of generators of the form
	\begin{gather*}
		g_{a_1} * g_{a_2} *\cdots * h_{b_1} * h_{b_2}* \cdots * e_{1,c_{1,1}} * e_{1,c_{1,2}} *\cdots  * e_{n,c_{n,1}} * \cdots * f_{1, d_{1,1}} *\cdots * f_{n, d_{n,1}}*  \cdots 
	\end{gather*}
	with~$a_1\geq a_2\geq \cdots$,~$b_1 \geq b_2 \geq \cdots$,~$c_{k,1} > c_{k, 2} > \cdots$, and~$d_{k, 1} > d_{k, 2} > \cdots$.
	It follows that as a~$\Z\times \Lambda_n$-graded~$\Q$-vector space, we have~$\Pp_n \cong_{\mathrm{gr\text{-}vsp}} \Sym(V\otimes \Q[z])$ where~$\deg(z)=(2,0)$ and
	\begin{equation*}
		V=\C^-_{\mathrm{e}_1}[1] \oplus \C^-_{\mathrm{e}_2}[1] \oplus \cdots \oplus \C^-_{\mathrm{e}_n}[1] \oplus \C^-_{\mathrm{f}_1}[1] \oplus \dots \oplus \C^-_{\mathrm{f}_n}[1] \oplus \C^+_\dnull[0] \oplus \C^+_\dnull[2].
	\end{equation*}
	Here, the subscript is the dimension degree, the number in brackets is the cohomological shift, and the sign (which is the cohomological degree mod~$2$) denotes parity of the summand as a super vector space.
\end{prop}

\begin{rk}
	For~$n\geq 1$, the generators~$h_{2i+2} = \{ e_{1,1}, f_{1,2i+1}\}$ are extraneous.
	It turns out that all the relations involving commutators with~$h_{2i+2}$ already follow from the other relations.
\end{rk}

\section{Computations of Relations in Small CoHAs}\label{section6}

By \Cref{CanAlgsForSmalln}, the canonical algebra~$C(\boldsymbol{\lambda}; 2^n)$ is given by a quiver without any relations for~$n\leq 2$.
We now use the methods of \cite{FRChowHa} to construct well-defined algebra homomorphisms from~$\Pp_n$ to~$\Coha(\Pbb^1(2^n))$ for~$n=0,\, 1, \, 2$.

\subsection{Computations in CoHAs of Quivers with Stability}

Given a quiver~$Q$, a stability function~$\theta \colon \Z^{Q_0} \to \Z$, and a slope~$\mu_0 \in \Q$, we can consider the category~$\Rep^{\tsst, \mu_0}(Q)$ of~$\theta$-semistable representations with slope~$\mu_0$.
Let
\[
\Lambda^{\theta, \mu_0}_+ \coloneqq \left\{ \dimvectd \in \N_0^{Q_0} \ \middle\vert \ {\theta(\dimvectd)}/{\sum_i d_i} = \mu_0 \right\} 
\]
be the monoid of dimension vectors of slope~$\mu_0$.
The stack~$\Mf$ of objects of this category is the disjoint union~$\Mf= \coprod_{\dimvectd \in \Lambda^{\theta, \mu_0}_+} \Mf_\dimvectd$ with~$\Mf_\dimvectd = [R_\dimvectd^\sst(Q) / G_\dimvectd]$, where
\[
R_\dimvectd^\sst(Q) \subseteq R_\dimvectd(Q) = \prod_{i \xrightarrow{\alpha} j} \mathbb{A}^{d_j \times d_i}
\]
is the open subset of~$\theta$-semistable representations and~$G_\dimvectd = \prod_{i} \GL_{d_i}(\C)$.

We can now construct the CoHA and the ChowHA of the category~$\Rep^{\tsst, \mu_0}(Q)$, where we note that~$H^\bullet(\Mf_\dimvectd) = H^\bullet_{G_\dimvectd}(R_\dimvectd^\sst(Q))$.
By \cite[Theorem~5.1]{FRChowHa}, the CoHA and the ChowHA coincide.

We use the following two propositions to do calculations.

\begin{prop}[{\cite[Theorem~2.2]{KSCoHADef}}] \label{Quiver product explicit}
	We have
	\[
	H^\bullet_{G_{\dimvectd}}(R_\dimvectd(Q)) = \Q[x_{i, j} \ \vert \ i \in Q_0,\, j = 1, \dots, d_i]^{S_\dimvectd}.
	\]
	For~$f(x) \in H^\bullet_{G_{ \dimvectd }}(R_{ \dimvectd } (Q) )$ and~$g(x) \in H^\bullet_{G_{\dimvecte}}(R_{\dimvecte} (Q) )$, the CoHA multiplication is given by
	\begin{align*}
		f*g = \sum_{(\sigma_i)_{i \in Q_0}}  & f(x_{i, \sigma_i(r)} \ \vert \ i ,\, 1 \leq r \leq d_i) \cdot g(x_{i, \sigma_i(d_i + s)} \ \vert \ i ,\, 1 \leq s \leq e_i)  \\ 
		& \cdot \prod_{i,j \in Q_0} \prod_{r= 1}^{d_i} \prod_{s = 1}^{e_j} (x_{j, \sigma_j(d_j + s)} - x_{i, \sigma_i(r)} ) ^{a_{i,j}- \delta_{i,j}},
	\end{align*}
	where the sum runs over all~$Q_0$-tuples of~$(d_i,e_i)$--shuffle permutations~$\sigma_i$.
	Here,~$a_{i,j}$ is the number of arrows from~$i$ to~$j$ in~$Q$ and~$\delta_{i,j}$ is the Kronecker delta. 
\end{prop}

\begin{prop}[{\cite[Theorem~8.1]{FRChowHa}}] \label{Quiver Tautological Presentation}
	The map~$H^\bullet_{G_\dimvectd}(R_\dimvectd(Q)) \to H^\bullet_{G_\dimvectd}(R_\dimvectd^\sst(Q))$ is surjective (because the cycle map is an isomorphism) and the kernel is given by
	\[
	\sum_{\substack{\dimvectd = \dimvectd_1 \! + \dimvectd_2, \\ 		\mu(\dimvectd_1) > \mu(\dimvectd_2)}} H^\bullet_{G_{\dimvectd_1}}(R_{\dimvectd_1}(Q)) * H^\bullet_{G_{\dimvectd_2}}(R_{\dimvectd_2}(Q)),
	\]
	where~$\mu(\dimvectd) = \theta(\dimvectd) / \sum_{i\in Q_0} d_i$.
\end{prop}

\subsection{Relations in the CoHA of the~\texorpdfstring{$2$}{2}-Kronecker Quiver}

The canonical algebra~$C(\lambda_1, \dots, \lambda_n ; w_1, \dots, w_n)$ for~$n=0$ is the path algebra of the Kronecker quiver~$K_2 =\begin{tikzcd}
	0 \rar[bend left] \rar[bend right] & \infty
\end{tikzcd}$ and regular representations are exactly the~$\theta$-semistable slope~$0$ representations for stability function~$\theta(\dimvectd) = \langle \dnull, \dimvectd\rangle = d_0 - d_\infty$, where~$\dnull = (1,1)$.

In \cite{FRKronecker} it has been shown that the CoHA of this category is isomorphic to the algebra~$\Pp_0$ from \Cref{Section alg Ppn}.
We recall a part of this calculation here.

We write~$H^\bullet_{G_\dimvectd}(R_\dimvectd(K_2)) = \Q[x_1,\dots, x_{d_0}, z_{1}, \dots, z_{d_\infty}]^{S_{d_0} \times S_{d_\infty}}$.
Using \Cref{Quiver product explicit} and \Cref{Quiver Tautological Presentation}, one can show:

\begin{la}[{\cite[Lemma~2]{FRKronecker}}]
	We have
	\[
	H^\bullet(\Mf_{(1,1)}) = H^\bullet_{G_{(1,1)}}(R_{(1,1)}^\sst(K_2)) = \Q[x,z] / (z-x)^2.
	\]
\end{la}

\begin{mydef}
	We set
	\begin{equation*}
		g_{2i}   \coloneqq x^i \in H^{2i}_{G_{(1,1)}}(R^\sst_{(1,1)}(K_2)) \quad \text{and} \quad 
		h_{2i+2}   \coloneqq x^i(z-x) \in H^{2i+2}_{G_{(1,1)}}(R^\sst_{(1,1)}(K_2)).
	\end{equation*} 
\end{mydef}

Note that these elements form a homogeneous~$\Q$-basis of~$H^{\bullet}_{G_{(1,1)}}(R^\sst_{(1,1)}(K_2))$.

\begin{rk}
	In \cite{FRKronecker}, these elements were called~$e_i$ and~$f_{i+1}$, respectively.
	We choose to use the index to denote the grading of the cohomology instead of the Chow groups.
\end{rk}

The following is an application of \Cref{Quiver product explicit} and \Cref{Quiver Tautological Presentation}.

\begin{prop}[{\cite[Lemma~3]{FRKronecker}}] \label{A1tilde Pp0 to Coha}
	We have a well-defined morphism of algebras
	\begin{align}\label{A1tilde Pp0 to Coha map equation}
		\Pp_0 \to \Coha(\Pbb^1) = \Coha(\Pbb^1(2^0))
	\end{align}
	mapping the generators~$h_{2i+2}$, and~$g_{2i}$ of~$\Pp_0$ to the corresponding elements in~$\Coha(\Pbb^1)$.
\end{prop}

It can now be shown that the elements~$g_{2i}$ and~$h_{2i+2}$ generate the CoHA and that the Poincaré series of the CoHA coincides with the one of~$\Pp_0$, thus showing that the map~\eqref{A1tilde Pp0 to Coha map equation} is an isomorphism.
We give these arguments in our more general situation of~$C(\boldsymbol{\lambda}; 2^n)$ in \Cref{Section Coha = Ppn}.

\subsection{Relations in the CoHA of \texorpdfstring{$\tilde{A}_2$}{A2tilde}}

Consider the acyclic~$\tilde{A}_2$ quiver
\[
Q_1 = \begin{tikzcd}[column sep = tiny, row sep = tiny]
	& 1 \ar{dr} & \\
	0 \ar{rr} \ar{ur} & & \infty.
\end{tikzcd}
\]

This quiver is (equivalent to) the canonical algebra~$C([1:0]; 2^1)$ by \Cref{CanAlgsForSmalln} and regular representations are the same as slope~$0$ semistable representations of stability~$\langle \dnull, -\rangle$, where~$\dnull$ is the dimension vector with all~$1$.
We also write~$\mathrm{e}_1$ for the dimension vector~$(0,1,0)$ and~$\mathrm{f}_1 \coloneqq \dnull - \mathrm{e}_1 = (1,0,1)$.

\begin{la}\label{A2tilde explicit H101 H010 H111}
	We have
	\begin{itemize}
		\item$H^\bullet(\Mf_{\mathrm{e}_1}) = H^\bullet_{G_{(0,1,0)}}(R^\sst_{(0,1,0)}(Q_1)) = \Q[y]$,
		\item $H^\bullet(\Mf_{\mathrm{f}_1}) = H^\bullet_{G_{(1,0,1)}}(R^\sst_{(1,0,1)}(Q_1)) = \Q[x,z]/(z-x)$, and 
		\item $H^\bullet(\Mf_{{\dnull}}) = H^\bullet_{G_{(1,1,1)}}(R^\sst_{(1,1,1)}(Q_1)) = \Q[x,y,z]/((z-x)(z-y),(z-x)(y-x))$.
	\end{itemize}
	Here, all variables have cohomological degree~$2$.
\end{la}
\begin{proof}
	Direct calculation using \Cref{Quiver product explicit} and \Cref{Quiver Tautological Presentation}.
\end{proof}

\begin{mydef}
	We set 
	\begin{align*}
		e_{2i+1} & \coloneqq y^i\in H^{2i+1}(\Mf_{\mathrm{e}_1};\Q^\mathrm{vir}) = H^{2i}_{G_{(0,1,0)}}(R^\sst_{(0,1,0)}(Q_1)), \\
		f_{2i+1} & \coloneqq x^i \in H^{2i+1}(\Mf_{\mathrm{f}_1};\Q^\mathrm{vir}) =  H^{2i}_{G_{(1,0,1)}}(R^\sst_{(1,0,1)}(Q_1)), \\
		g_{2i} & \coloneqq x^i \in H^{2i}(\Mf_{\dnull};\Q^\mathrm{vir}) = H^{2i}_{G_{(1,1,1)}}(R^\sst_{(1,1,1)}(Q_1)), \\
		h_{2i+2} & \coloneqq (z-x)x^i\in H^{2i+2}(\Mf_{\dnull}; \Q^{\mathrm{vir}}) = H^{2i+2}_{G_{(1,1,1)}}(R^\sst_{(1,1,1)}(Q_1)).
	\end{align*}
\end{mydef}

\begin{cor}\label{A2tilde basis for e1 and f1 degree}
	The elements~$e_{2i+1}$ form a (graded)~$\Q$-basis of~$H^\bullet(\Mf_{(0,1,0)})$ and the elements~$f_{2i+1}$ of~$H^\bullet(\Mf_{(1,0,1)})$.	
\end{cor}

\begin{prop}\label{A2tilde Pp1 to Coha}
	We have a well-defined morphism of algebras
	\begin{align}\label{A2tilde Pp1 to Coha map equation}
		\Pp_1 \to \Coha(\Pbb^1(2^1)) = \Coha^{\sst, 0}(Q_1),
	\end{align}
	mapping the generators~$e_{1, 2i+1}$,~$f_{1, 2i+1}$,~$h_{2i+2}$, and~$g_{2i}$ of~$\Pp_1$ to the corresponding elements in~$\Coha(\Pbb^1(2^1))$.
\end{prop}

\begin{proof}
	We need to verify the commutator-relations between the images of the generators of~$\Pp_1$.
	We write the relations in terms of generating series as in \Cref{section5}.
	Also abbreviate~$\tilde{H}(X,Y) = \frac{YH(Y) - XH(X)}{Y - X}$.
	
	The relations involving only~$G(X)$ and~$H(X)$ follow from the corresponding relations in~$\Coha(\Pbb^1(2^0)) = \Pp_0$ and functoriality \Cref{thm functorial}.
	
	For the relation~$[G(X), F(Y)] = (YF(Y) - XF(X)) *\tilde{H}(X,Y)$, we calculate
	\begin{small}
		\begin{align*}
			F(&Y)Y*\tilde{H}(X,Y)  = \frac{1}{1-xY}Y* \frac{z-x}{(1-xX)(1-xY)} \\
			& = \frac{1}{(x_2-x_1)(z_2-z_1)}\left(\begin{aligned}
				& \frac{1}{(1-x_1Y)(1-x_2X)(1-x_2Y)}(z_2-x_2)(z_2-x_1)(y-x_1)Y \\
				- & \frac{1}{(1-x_1Y)(1-x_2X)(1-x_2Y)}(z_1-x_2)(z_1-x_1)(y-x_1)Y \\
				- & \frac{1}{(1-x_2Y)(1-x_1X)(1-x_1Y)}(z_2-x_1)(z_2-x_2)(y-x_2) Y\\
				+ & \frac{1}{(1-x_2Y)(1-x_1X)(1-x_1Y)}(z_1-x_1)(z_1-x_2)(y-x_2) Y			
			\end{aligned}\right) \\
			& = \frac{(z_1+z_2)-(x_1+x_2)}{x_2-x_1} \left(\begin{aligned}
				& \frac{1}{(1-x_2X)(1-x_2Y)} - \frac{1-yY}{(1-x_1Y)(1-x_2X)(1-x_2Y)} \\
				- & \frac{1}{(1-x_1X)(1-x_1Y)} + \frac{1-yY}{(1-x_2Y)(1-x_1X)(1-x_1Y)}
			\end{aligned}\right). 
		\end{align*}
	\end{small}
	It follows that
	\begin{align*}
		(F(Y) & Y-F(X)X) * \tilde{H}(X,Y)  \\
		& =  \frac{(z_1+z_2)-(x_1+x_2)}{x_2-x_1} \left(\begin{aligned}
			& - \frac{1-yY}{(1-x_1Y)(1-x_2X)(1-x_2Y)} \\ & + \frac{1-yY}{(1-x_2Y)(1-x_1X)(1-x_1Y)} \\
			& + \frac{1-yX}{(1-x_1X)(1-x_2Y)(1-x_2X)} \\ & - \frac{1-yX}{(1-x_2X)(1-x_1Y)(1-x_1X)} 
		\end{aligned}\right) \\
		& = \frac{(z_1+z_2)-(x_1+x_2)}{(1-x_1X)(1-x_2X)(1-x_1Y)(1-x_2Y)} \cdot (Y-X).
	\end{align*}
	
	We also have
	\begin{align*}
		G(X)*F(Y) & = \frac{1}{(x_2-x_1)(z_2-z_1)}\left(\begin{aligned}
			& \frac{1}{(1-x_1X) (1-x_2Y)}(z_2-x_1)(z_2-y) \\
			- & \frac{1}{(1-x_1X) (1-x_2Y)}(z_1-x_1)(z_1-y) \\
			- & \frac{1}{(1-x_2X) (1-x_1Y)}(z_2-x_2)(z_2-y) \\
			+ & \frac{1}{(1-x_2X) (1-x_1Y)}(z_1-x_2)(z_1-y) 
		\end{aligned}
		\right) \\
		& = \frac{1}{x_2-x_1}\left(\begin{aligned}
			& \frac{1}{(1-x_1X) (1-x_2Y)}(z_1+z_2-x_1-y) \\
			- & \frac{1}{(1-x_2X) (1-x_1Y)}(z_1+z_2-x_2-y)
		\end{aligned}\right)\\
		\intertext{and similarly}
		F(Y)*G(X) & = \mycomment{\frac{1}{(x_2-x_1)(z_2-z_1)}\left(\begin{aligned}
				& \frac{1}{(1-x_1Y)(1-x_2X)}(z_2-x_1)(y-x_1) \\
				- & \frac{1}{(1-x_1Y)(1-x_2X)}(z_1-x_1)(y-x_1) \\
				- & \frac{1}{(1-x_2Y)(1-x_1X)}(z_2-x_2)(y-x_2) \\
				+ & \frac{1}{(1-x_2Y)(1-x_1X)}(z_1-x_2)(y-x_2) 
			\end{aligned}\right) \\
			& = }\frac{1}{x_2-x_1}\left(\begin{aligned}
			& \frac{1}{(1-x_1Y)(1-x_2X)}(y-x_1) \\
			- & \frac{1}{(1-x_2Y)(1-x_1X)}(y-x_2) \\
		\end{aligned}\right).
	\end{align*}
	It follows that
	\begin{align*}
		[G(X),F(Y)] & = \frac{1}{x_2-x_1}\left(\begin{aligned}
			& \frac{1}{(1-x_1X)(1-x_2Y)}((z_1+z_2)-(x_1+x_2)) \\ 
			- & \frac{1}{(1-x_2X)(1-x_1Y)}((z_1+z_2)-(x_1+x_2))
		\end{aligned}\right) \\
		& = \frac{(z_1+z_2)-(x_1+x_2)}{(1-x_1X)(1-x_2X)(1-x_1Y)(1-x_2Y)}\cdot(Y-X) \\
		& = (F(Y)Y-F(X)X)*\tilde{H}(X,Y).
	\end{align*}
	
	We omit the calculations for the remaining relations.
\end{proof}

\subsection{Relations in the CoHA of \texorpdfstring{$\tilde{A}_3$}{A3tilde}}
We consider now the affine~$\tilde{A}_3$ quiver 
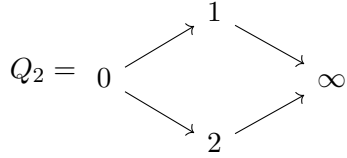
\[
Q_2 = {\begin{tikzcd}[ampersand replacement =\&, row sep = small]
		\&  1 \ar{dr} \& \\
		0 \ar{ur}\ar{dr} \& \&  \infty \\
		\& 2 \ar{ur} \& 
\end{tikzcd}}
\]
with stability~$\theta(\dimvectd) = \langle \dnull , \dimvectd \rangle = d_0 - d_\infty$ for~$\dnull = (1,1,1,1)$.

The quiver~$Q_2$ is (equivalent to) the canonical algebra~$C([1 : 0], [0 : 1]; 2,2)$ and regular representations are the same as semistable representations of slope~$0$.
We again denote~$\mathrm{e}_1 \coloneqq (0,1,0,0)$ and~$\mathrm{e}_2 \coloneqq (0 , 0 , 1 , 0)$, as well as~$\mathrm{f}_1 \coloneqq \dnull - \mathrm{e}_1 = (1, 0 , 1 , 1)$ and~$\mathrm{f}_2 \coloneqq \dnull - \mathrm{e}_2 = ( 1 , 1 , 0 , 1)$.

\begin{la}\label{A3tilde cohomology rings for d < d0}
	We have
	\begin{itemize}
		\item $H^\bullet(\Mf_{\mathrm{e}_1}) = H^\bullet_{G_{(0,1,0,0)}}(R^\sst_{(0,1,0,0)}(Q_2)) = \Q[u]$,
		\item $H^\bullet(\Mf_{\mathrm{e}_2}) = H^\bullet_{G_{(0,0,1,0)}}(R^\sst_{(0,0,1,0)}(Q_2)) = \Q[v]$,
		\item $H^\bullet(\Mf_{\mathrm{f}_1}) = H^\bullet_{G_{(1,0,1,1)}}(R^\sst_{(1,0,1,1)}(Q_2)) = \Q[x,v,z]/(x-v,v-z)$, and
		\item $H^\bullet(\Mf_{\mathrm{f}_2}) = H^\bullet_{G_{(1,1,0,1)}}(R^\sst_{(1,1,0,1)}(Q_2)) = \Q[x,u,z]/(x-u,u-z)$.
	\end{itemize}
\end{la}
\begin{proof}
	Follows immediately from \Cref{Quiver product explicit} and \Cref{Quiver Tautological Presentation}.
\end{proof}


\begin{mydef}
	We set 
	\begin{align*}
		e_{1, 2i+1} & \coloneqq u^i\in H^{2i+1}(\Mf_{\mathrm{e}_1}; \Q^\mathrm{vir}) = H^{2i}(\Mf_{\mathrm{e}_1}),  & e_{2, 2i+1} & \coloneqq v^i\in H^{2i+1}(\Mf_{\mathrm{e}_2}; \Q^\mathrm{vir}), \\
		f_{1, 2i+1} & \coloneqq x^i \in H^{2i+1}(\Mf_{\mathrm{f}_1}; \Q^\mathrm{vir}) = H^{2i}(\Mf_{\mathrm{f}_1}), & 
		f_{2, 2i+1} & \coloneqq x^i \in H^{2i+1}(\Mf_{\mathrm{f}_2}; \Q^\mathrm{vir}) , \\
		g_{2i} & \coloneqq x^i \in H^{2i}(\Mf_{\dnull}; \Q^\mathrm{vir}) = H^{2i}(\Mf_\dnull), &
		h_{2i+2} & \coloneqq (z-x)x^i \in H^{2i + 2}(\Mf_{\dnull}; \Q^\mathrm{vir}).
	\end{align*}
\end{mydef}

The functor~$\Phi_1 \colon \Rep(Q_2) \to \Rep(Q_1)$ given by
\[
\begin{tikzcd}[column sep = tiny, row sep = tiny]
	& \C^{d_1} \ar{dr}{A_2} & \\
	\C^{d_0}  \ar{ur}{A_1}\ar{dr}{B_1} & & \C^{d_\infty} \\
	& \C^{d_2} \ar{ur}{B_2} & 
\end{tikzcd} \mapsto 	\begin{tikzcd}[column sep = tiny, row sep = tiny]
	& \C^{d_1} \ar{dr}{A_2} & \\
	\C^{d_0}  \ar{ur}{A_1} & &\ar[leftarrow]{ll}{B_2\cdot B_1} \C^{d_\infty}  
\end{tikzcd} 
\]
restricts to~$\Phi_1\colon  \Rep^{\sst}(Q_2) \to \Rep^{\sst}(Q_1)$ and induces a morphism of graded vector spaces
\[
\Coha^{\sst}(Q_1) \to \Coha^{\sst}(Q_2),
\]
which is an algebra homomorphism by \Cref{thm functorial}.

This morphism maps
\begin{align*}
	e_{2i+1 }\in \Coha^{\sst}(Q_1) \quad & \text{to} \quad e_{1, 2i+2 }\in \Coha^{\sst}(Q_2), \\
	f_{2i+1 }\in \Coha^{\sst}(Q_1) \quad & \text{to} \quad f_{1, 2i+2 }\in \Coha^{\sst}(Q_2), \\
	h_{2i+2 }\in \Coha^{\sst}(Q_1) \quad & \text{to} \quad h_{ 2i + 2 }\in \Coha^{\sst}(Q_2),\quad \text{and} \\
	g_{ 2i }\in \Coha^{\sst}(Q_1) \quad & \text{to} \quad g_{ 2i }\in \Coha^{\sst}(Q_2).
\end{align*}

We get a similar functor and thus a morphism between the CoHAs if we contract the other arm.

\begin{prop}\label{A3tilde Pp2 to Coha}
	We have a well-defined morphism of algebras
	\begin{align}\label{A3tilde Pp2 to Coha map equation}
		\Pp_2 \to \Coha(\Pbb^1(2^2)) = \Coha^{\sst,0}(Q_2),
	\end{align}
	mapping the generators~$e_{k, 2i+1}$,~$f_{k, 2i+1}$ for~$k = 1, 2$,~$h_{2i+2}$, and~$g_{2i}$ of~$\Pp_2$ to the corresponding elements in~$\Coha^\sst(Q_2)$.
\end{prop}
\begin{proof}
	Almost all relations follow from the corresponding relations in~$\Coha(\Pbb^1(2^1))$ and functoriality.
	The only relations that are remaining are
	\begin{gather*}
		[E_1(X),E_2(Y)] = [E_1(X) , F_2(Y) ] = [ F_1(X) , F_2(Y) ] = 0.
	\end{gather*}
	This is a direct calculation, similar to the one before.
\end{proof}

\section{CoHA of Weighted Projective Lines}\label{Section Coha = Ppn} \label{section7}

In this section, we finally come to the main result of this thesis.
We express the algebra~$\Coha(\Pbb^1(2^n))$ -- i.e., the cohomological Hall algebra of the hereditary abelian category~$\Tor(\Pbb^1(\boldsymbol{\lambda}; 2^n)) = \Reg(C(\boldsymbol{\lambda}; 2 ^n))$ -- in terms of generators and relations:

\begin{thm}\label{CohaWPL(2^n)=Ppn}
	We have~$\Coha(\Pbb^1(2^n)) \cong \Pp_n$ as~$\Z\times \Lambda_n^+$-graded~$\Q$-algebras, where~$\Pp_n$ is the algebra defined in \Cref{section5}.
\end{thm}

Before we come to the proof, we have the following corollaries:

\begin{cor}
	The CoHA of regular representations of the type~$\tilde{A}_3$ quiver
	
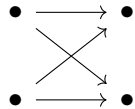
\[
	\begin{tikzcd}
		\bullet \rar \drar & \bullet \\
		\bullet \rar \urar & \bullet
	\end{tikzcd}
	\]
	is given by~$\Pp_2$.
\end{cor}
\begin{proof}
	The category of regular representations of
	
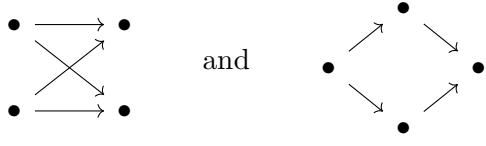
\[
	\begin{tikzcd}
		\bullet \rar \drar & \bullet \\
		\bullet \rar \urar & \bullet
	\end{tikzcd} \qquad \text{and} \qquad 
	\begin{tikzcd}[column sep = small, row sep = small]
		& \bullet \drar & \\
		\bullet \urar \drar & & \bullet \\
		& \bullet \urar & 
	\end{tikzcd}
	\]	
	are related by reflection functors and thus induce isomorphic CoHAs, see \cite[Section~3]{astruc2024motivescentralslopekronecker} or \cite[Proposition~5.2]{KSCoHADef}.
\end{proof}
\begin{cor}
	The CoHA of regular representations of any orientation of a~$\tilde{D}_4$ quiver is given by~$\Pp_3$.
\end{cor}
\begin{proof}
	By reflection functors, the CoHA is independent of the orientation of~$\tilde{D}_4$.
	It is known that~$\Rep(\tilde{D}_4)$ is tilting equivalent to~$\Rep(C(2^3))$, see \cite[Chapter~5]{RingelSLN1099}, \cite[Section~5.4.1]{GeigleLenzingWPCarisinginRepTh}, or more explicitly \cite{KussinMeltzerIndecsForDomesticCanAlgs}, and that this equivalence induces an equivalence between regular representations.
\end{proof}

\subsection{Construction of the Map}

Recall that by \Cref{thm functorial} and \Cref{A=H for Pn} we have a graded algebra homomorphism
\begin{equation} \label{CohaFunctorialMorphism}
	\Coha(\Pbb^1(2^n)) \to \Coha(\Pbb^1(2^{n+1}))
\end{equation}
for every~$n\in \N_0$.

\begin{prop}\label{Coha injection}
	The map~$\Coha(\Pbb^1(2^{n}))\to \Coha(\Pbb^1(2^{n+1}))$ is injective.
\end{prop}
\begin{proof}
	By \Cref{open substack of top stack isomorphic to bottom}, we have that
	\[
	\Mf_{\tilde{\dimvectd}}^{\mathrm{inv}}(2^{n+1}) \mono \Mf_{\tilde{\dimvectd}}(2^{n+1}) \to \Mf_\dimvectd(2^n)
	\]
	is an isomorphism.
	It follows that the restriction map~$H^\bullet(\Mf_{\tilde{\dimvectd}}(2^{n+1}) ) \to H^\bullet( \Mf_{\tilde{\dimvectd}}^{\mathrm{inv}}(2^{n+1}) )$ is a 
	post-split for~$H^\bullet(\Mf_\dimvectd(2^{n}))\to H^\bullet(\Mf_{\tilde{\dimvectd}}(2^{n + 1}))$.
\end{proof}

In \Cref{section6}, we considered~$\Coha(\Pbb^1(2^n))$ for~$n= 0,\, 1,\, 2$ and defined certain elements in these algebras.
We now lift these elements along the maps~\eqref{CohaFunctorialMorphism} to elements of~$\Coha(\Pbb^1(2^n))$ for every~$n\geq 0$.

\begin{mydef}
	Consider the (injective) algebra homomorphisms
	\[
	\Phi_k \colon \Coha(\Pbb^1(2^1))\to \Coha(\Pbb^1(2^n)) \quad \text{for~$k = 1, \dots, n$}	
	\]
	induced by the functors~$\Reg(C(2^n)) \to \Reg(C(2^1))$ which forgets all but the~$k$-th arm, and
	\[
	\Phi_0\colon \Coha(\Pbb^1) \to \Coha(\Pbb^1(2^n))
	\]
	induced by the functor~$\Reg(C(2^n)) \to \Reg(C(2^0)) = \Rep^\sst(K_2)$.
	We define the elements
	\begin{align*}
		e_{k, 2i+1} & \in \Coha(\Pbb^1(2^n))_{ (2i+1, \mathrm{e}_k)},  &	f_{k, 2i+1} & \in \Coha(\Pbb^1(2^n))_{ (2i+1, \mathrm{f}_k)}, \\
		h_{2i+2} & \in \Coha(\Pbb^1(2^n))_{(2i+2, \dnull)}, & g_{2i} & \in \Coha(\Pbb^1(2^n))_{(2i, \dnull)} \\
		\intertext{as the images of the respective elements}
		e_{2i+1} & \in \Coha(\Pbb^1(2^1))_{ (2i+1, \mathrm{e}_k)},  &	f_{2i+1} & \in \Coha(\Pbb^1(2^1))_{ (2i+1, \mathrm{f}_k)}, \\
		h_{2i+2} & \in \Coha(\Pbb^1)_{(2i + 2, \dnull)}, & g_{2i} & \in \Coha(\Pbb^1)_{(2i, \dnull)}
	\end{align*}
	under these homomorphisms.
\end{mydef}


\begin{prop}\label{PpnToCohaProp}
	The elements~$e_{k, 2i+1}$,~$f_{k, 2i+1}$,~$h_{2i+2}$,~$g_{2i}$ satisfy the defining relations of~$\Pp_n$.
	We thus have a well-defined morphism of~$\Z\times \Lambda_n^+$-graded~$\Q$-algebras
	\begin{align}\label{PpnToCohaMorphism}
		\Pp_n \to \Coha(\Pbb^1(2^n)),
	\end{align}
	mapping the generators~$e_{k, 2i+1}$,~$f_{k, 2i+1}$,~$h_{2i+2}$,~$g_{2i}\in \Pp_n$ to the corresponding elements in~$\Coha(\Pbb^1(2^n))$.
\end{prop}
\begin{proof}
	This follows from the calculation of the relations in the algebras~$\Coha(\Pbb^1(2^1))$, and~$\Coha(\Pbb^1(2^2))$ in \Cref{A2tilde Pp1 to Coha} and \Cref{A3tilde Pp2 to Coha}, respectively.
\end{proof}


\subsection{Surjectivity}

\begin{prop}\label{stable reps of canonical algebra}
	The only simple regular (= stable of slope~$0$) representations of~$C(2^n)$ appear in dimensions~$\mathrm{e}_k = (0,0,\dots, 1,0,\dots, 0,0)$,~$\mathrm{f}_k = (1,1,\dots, 0, \dots, 1,1) = \dnull - \mathrm{e}_k$ for~$k = 1,\dots, n$, and in~$\dnull = (1,1,\dots, 1,1)$.
	We have
	\begin{gather}
		R^\simp_{\mathrm{e}_k}  = R_{\mathrm{e}_k} = \pt, \, [R^\simp_{\mathrm{e}_k}/G_{\mathrm{e}_k}] = {B}\mathbb{G}_m, \, \text{and } \mathcal{M}_{\mathrm{e}_k} = \pt, \label{stableRepsOfCanAlgs:ek} \\
		R^\simp_{\mathrm{f}_k}  = R_{\mathrm{f}_k} = (\C^*)^n,\,  [R^\simp_{\mathrm{f}_k}/G_{\mathrm{f}_k}] = B\mathbb{G}_m, \, \text{and } \mathcal{M}_{\mathrm{f}_k} = \pt, \label{stableRepsOfCanAlgs:fk} \\
		\begin{gathered}
			R^\simp_{\dnull}  = \{ (\alpha, \beta) \ \vert \  [ \alpha : \beta ] \neq \lambda_i  \} \times (\C^*)^{n-1},\, \\ [ R^\simp_{ \dnull } / G_{ \dnull } ] = \mathbb{P}^1\setminus \{\lambda_1 , \dots, \lambda_n \} \times B\mathbb{G}_m, \, \text{and } \mathcal{M}_\dnull = \Pbb^1, 
		\end{gathered}\label{stableRepsOfCanAlgs:d0}
	\end{gather}
	\mycomment{
		\begin{itemize}
			\item $R^\simp_{\mathrm{e}_k} = R_{\mathrm{e}_k} = \pt$,~$[R^\simp_{\mathrm{e}_k}/G_{\mathrm{e}_k}] = {B}\mathbb{G}_m$, and~$\mathcal{M}_{\mathrm{e}_k} = \pt$,
			\item $R^\simp_{\mathrm{f}_k} = R_{\mathrm{f}_k} = (\C^*)^n$,~$[R^\simp_{\mathrm{f}_k}/G_{\mathrm{f}_k}] = B\mathbb{G}_m$, and~$\mathcal{M}_{\mathrm{f}_k} = \pt$,
			\item $R^\simp_{\dnull} = \{ M \in R_{(1,1)}^\mathrm{simp}(K_2) \ \vert \  [\alpha:\beta] \neq \lambda_i  \} \times (\C^*)^{n-1}$,~$[ R^\simp_{ \dnull } / G_{ \dnull } ] = \mathbb{P}^1\setminus \{\lambda_1 , \dots, \lambda_n \} \times B\mathbb{G}_m$, and~$\mathcal{M}_\dnull = \Pbb^1$,
	\end{itemize}}
	where~$\mathcal{M}_\dimvectd = R_\dimvectd \gitquotient PG_\dimvectd$ is the course moduli space parametrizing semisimple objects, see \cite[Section~3.2]{meinhardt2015donaldsonthomasinvariantsvsintersection}.
\end{prop}

\begin{proof}
	Any simple regular representation~$M = (\alpha, \beta, A_1^{(1)} \! , \dots, A_n^{(2)})$ will be sent either to a simple regular representation of~$K_2$ or to the zero representation under the natural functor~$M \mapsto (\alpha, \beta)$.
	It follows that the dimension vector of~$M$ will satisfy~$d_0 = d_\infty \in \{ 0,1\}$.
	If~$d_0 = 0$, then we can find a subrepresentation supported on a single middle vertex and so~$M$ is of dimension~$\mathrm{e}_{k}$.
	If~$d_0 = 1$, then we also see that~$d_{k}\leq 1$.
	If we had~$d_k = d_{l}= 0$ for~$k\neq l$, then~$\lambda_k^{(1)} \alpha + \lambda_k^{(2)} \beta = \lambda_l^{(1)} \alpha + \lambda_l^{(2)} \beta = 0$ and so~$\alpha = \beta = 0$, because~$\lambda_k  \neq \lambda_l \in \Pbb^1(\C)$.
	It follows that~$(\alpha, \beta)$ and thus~$M$ is not regular.
	\mycomment{
		Simple regular representations of~$C(2^n)$ are the same as simple torsion sheaves on~$\Pbb^1(2^n)$ by \Cref{Tor(Pbb1) = Reg(C) stacky}.
		The simple objects in~$\Tor(\mathbb{P}(2^n))$ are the length~$1$ objects.
		The length~$1$ objects are either supported on an exceptional/stacky point or on a point of the open subscheme.
		For any non-stacky point, there is a unique simple object.
		This corresponds to a stable representation of dimension vector~$\dnull = (1,1,\dots,1,1)$.
		The stacky points have exactly~$2$ simples.
		These correspond to dimension vectors~$\mathrm{e}_k$, resp.~$\mathrm{f}_k$.}
	
	We have~$R^\mathrm{simp}_{\mathrm{e}_k} = R_{\mathrm{e}_k} = R_{\mathrm{e}_k}(Q(2^n)) = \pt$.
	Therefore~$[R^\mathrm{simp}_{\mathrm{e}_k}/G_{\mathrm{e}_k}] = [\pt/\C^*] = B\C^*$ and~$\mathcal{M}_{\mathrm{e}_k} = \pt$.
	
	In dimension vector~$\mathrm{f}_k$, every regular representation is also automatically simple regular, because~$\mathrm{f}_k$ is not the sum of two smaller regular dimension vectors.
	We deduce
	\begin{align*}
		R&^\mathrm{simp}_{\mathrm{f}_k}  = R_{\mathrm{f}_k} \\
		& =  \left\{ (\alpha, \beta, A_i^{(1)}\!, A_i^{(2)})_{i \neq k} \in R_{(1,1)}^\sst(K_2)\times (\C^*)^{2(n-1)} \ \middle\vert \ \begin{aligned}
			& \lambda_i^{( 0 ) } \alpha + \lambda_i^{ ( 1 ) } \beta = A_i^{ ( 2 ) } A_i^{ ( 1 ) } \! , \, i\neq k, \\
			& \text{and } \lambda_k^{( 0 ) } \alpha + \lambda_k^{ ( 1 ) } \beta = 0 
		\end{aligned}\right\} \\
		& \cong \{ (\alpha, \beta)\in R_{(1,1)}^\sst(K_2) \ \vert \ \lambda_k^{ ( 0 ) } \alpha + \lambda_k^{ ( 1 ) } \beta = 0 \} \times (\C^*)^{n-1}  \cong (\C^*)^n.
	\end{align*}
	We see that~$[R_{\mathrm{f}_k}^\simp/G_{\mathrm{f}_k}]= [\pt/\C^*]$ and~$\mathcal{M}_{\mathrm{f}_k} = \pt$.
	
	A regular representation of dimension~$\dnull$ is simple regular if and only if it contains no subrepresentation of dimension vector~$\mathrm{e}_k$ or~$\mathrm{f}_k$.
	Thus, we have
	\begin{align*}
		R^\simp_{\dnull} & = \{ M\in R_\dnull \mid A_i^{ ( 1 ) } \! , A_i^{ ( 2 ) } \neq 0\} \\
		& = \{ M= ( \alpha, \beta)  \in R_{(1,1)}^\mathrm{st}(K_2) \ \vert \ [\alpha:\beta]\neq \lambda_i \} \times (\C^*)^{n} \\
		& = \{ ( \alpha , \beta ) \in \C^2 \setminus \{0\} \ \vert \ [ \alpha : \beta ] \neq \lambda_i \} \times (\C^*)^{n}.
	\end{align*}
	We get~$[R_\dnull^\simp/G_\dnull] = \mathbb{P}^1\setminus \{\lambda_1,\dots, \lambda_n\} \times [\pt/\C^*]$.
	From this it follows that~$\mathcal{M}_\dnull = \Pbb^1$.
\end{proof}

\begin{cor}\label{Coha generating degrees}
	The~$\Q$-algebra~$\Coha(\Pbb^1(2^n))(=\Chowha(\Pbb^1(2^n)))$ is generated in dimension degrees~$\mathrm{e}_k$,~$\mathrm{f}_k$, and~$\dnull$.
\end{cor}

\begin{cor}\label{CohaEFIsPolyRing}
	We have~$H^\bullet(\Mf_{\mathrm{e}_k}) \cong \Q[x]$ and~$H^\bullet(\Mf_{\mathrm{f}_k}) \cong \Q[x]$, where the variable~$x$ is of degree~$2$.
\end{cor}

\begin{la}\label{CohaEFBasis}
	The elements~$(e_{k,2i+1})_{i=0,1,2,\dots}$ form a graded~$\Q$-basis of~$\Coha(\Pbb^1(2^n))_{\mathrm{e}_k}$ and the elements~$(f_{k,2i+1})_{i=0,1,2,\dots}$ form a graded~$\Q$-basis of~$\Coha(\Pbb^1(2^n))_{\mathrm{f}_k}$.
\end{la}
\begin{proof}
	The elements~$e_{2i+1}$, respectively~$f_{2i+1}$, form a graded~$\Q$-basis of~$\Coha(\Pbb^1(2^1))_{\mathrm{e}_k}$, respectively~$\Coha(\Pbb^1(2^1))_{\mathrm{f}_k}$, by \Cref{A2tilde basis for e1 and f1 degree}.
	The assertion follows by \Cref{Coha injection} and \Cref{CohaEFIsPolyRing}.
\end{proof}

\begin{la}\label{CohaD0Basis}
	The elements~$g_{2i}$,~$h_{2i+2}$,~$e_{k, 2i+1} * f_{k, 2j+1}$, and~$f_{k, 2i+1} * e_{k, 2j+1}$ generate~$H^\bullet(\Mf_{\dnull})$ as a~$\Q$-vector space.
\end{la}

\begin{proof}
	By \Cref{stable reps of canonical algebra} we have the identifications
	\[
	\Mf^{\mathrm{simp}}_{(1,1)}(2^0) \cong \Pbb^1 \times B\mathbb{G}_m \quad \text{ and }\quad  \Mf^{\mathrm{simp}}_{\dnull}(2^n) \cong \mathbb{P}^1\setminus \{\lambda_1,\dots, \lambda_n\} \times B\mathbb{G}_m
	\]
	and thus that $A_\bullet(\Mf^{\mathrm{simp}}_{(1,1)}(2^0))  \to A_\bullet(\Mf^{\mathrm{simp}}_{\dnull}(2^n))$ is surjective.
	The assertion now follows from \Cref{Chowha primitive part}.
\end{proof}


\begin{cor}\label{PpnToCohaSurj}
	The algebra morphism~$\Pp_n \to \Coha(\Pbb^1(2^n))$ is surjective.
\end{cor}

\subsection{Poincaré Series of \texorpdfstring{$\Coha(\Pbb^1(2^n))$}{Coha(P1(2n))}}


\begin{prop}\label{Coha Poincare series}
	We have an isomorphism~$\Coha(\Pbb^1(2^n))\cong_{\mathrm{gr\text{-}vsp}} \Sym(V\otimes \Q[z])$ of (abstract)~$\Z \times \Lambda_n^+$-graded vector spaces, where~$z$ has bidgree~$(2,0)$ and
	\begin{equation*}
		V=\Q^-_{\mathrm{e}_1}[1] \oplus \Q^-_{\mathrm{e}_2}[1] \oplus \cdots \oplus \Q^-_{\mathrm{e}_n}[1] \oplus \Q^-_{\mathrm{f}_1}[1] \oplus \dots \oplus \Q^-_{\mathrm{f}_n}[1] \oplus \Q^+_\dnull[0] \oplus \Q^+_\dnull[2].
	\end{equation*}
	Here, the subscript is the dimension degree, the number in brackets is the cohomological shift, and the sign denotes parity of the summand as a super vector space.
\end{prop}
\begin{proof}
	We write~$\mathcal{R}^\mathrm{Hdg} \coloneqq K_0(\MHS)[\mathbb{L}^{-1/2}, \,  (\mathbb{L}^N-1)^{-1} : N \in \Z_{\geq 1}]$. 
	We have by definition of DT-invariants for categories of homological dimension one (c.f.\ \Cref{Def:DT}) that
	\[
	\sum_{\dimvectd\in \Lambda_n^+} [ \Mf_{\dimvectd} ]^\mathrm{vir}  t^\dimvectd = \Sym\left(	\frac{1}{\mathbb{L}^{1/2} - \mathbb{L}^{ - 1/2}} \sum_{\dimvectd} \mathrm{DT}_{\dimvectd} t^\dimvectd \right) 
	\]
	as elements in~$\mathcal{R}^\mathrm{Hdg} [\![t_i : i \in Q_0 ]\!]$, where
	\[
	[\Mf_{\dimvectd}]^\mathrm{vir} = (-\mathbb{L}^{1/2})^{ - \dim \Mf_{\dimvectd}} [\Mf_{\dimvectd}] = (-\mathbb{L}^{1/2})^{\langle \dimvectd , \dimvectd \rangle} [\Mf_{\dimvectd}] \in \mathcal{R}^{\mathrm{Hdg}}. 
	\]
	By \cite[Theorem~5.6]{meinhardt2015donaldsonthomasinvariantsvsintersection}, the DT-invariants satisfy 
	\[
	\mathrm{DT}_{\dimvectd} = \begin{cases}
		\IC_c (\mathcal{M}_{\dimvectd} ;  \Q), & \text{if~$R_{\dimvectd}^\simp \neq \emptyset$}, \\
		0, & \text{otherwise}
	\end{cases}
	\]
	as elements of~$\mathcal{R}^{\mathrm{Hdg}}$. 
	By \Cref{stable reps of canonical algebra},~$\mathrm{e}_k$,~$\mathrm{f}_k$, and~$\dnull$ are the only dimension vectors with simple objects, and we have~$\mathcal{M}_{\mathrm{e}_k} = \mathcal{M}_{\mathrm{f}_k} = \pt$ and~$\mathcal{M}_\dnull \cong \Pbb^1$.
	These spaces are smooth projective and so compactly supported intersection homology coincides with usual cohomology.
	We therefore find
	\[
	\sum_{\dimvectd\in \Lambda_n^+} [ \Mf_{\dimvectd}]^\mathrm{vir} t^\dimvectd = \Sym\left(	\frac{1}{\mathbb{L}^{1/2} - \mathbb{L}^{ - 1/2}}   \left(\sum_{k = 1}^n \left(t^{\mathrm{e}_k} + t^{\mathrm{f}_k} \right) +  ( \mathbb{L}^{-1/2} + \mathbb{L}^{1/2})t^\dnull\right) \right).
	\]
	In a last step, we use purity of the cohomology of~$\Mf_{\dimvectd} = [R_{\dimvectd} / G_{\dimvectd}]$ (see \Cref{Purity of cohomology if affine paving} and \Cref{Affine stratification C(2^n)}) to deduce from \Cref{Poincare series of Coha = Sym(DT} that we can recover the Poincaré series of the CoHA from the class~$\sum_{\dimvectd\in \Lambda_n^+} [ \Mf_\dimvectd ]^\mathrm{vir}  t^\dimvectd \in \mathcal{R}^\mathrm{Hdg} [\![t_i : i \in Q_0 ]\!]$ and that we have
	\[
	P_{q,t}(\Coha(\Pbb^1(2 ^n))) = \Sym\left( \frac{1}{1-q^2}\left( \sum_{k=1}^n ( (-q)t^{\mathrm{e}_k} + (-q) t^{\mathrm{f}_k} ) + (1 + q^2) t^\dnull \right)\right). \qedhere
	\]
\end{proof}

\begin{rk}
	It should be possible to circumvent the use of \cite{meinhardt2015donaldsonthomasinvariantsvsintersection} and DT-invariants entirely by calculating the Poincaré series directly from the stratification of the stacks~$\Mf_{\dimvectd}$ from the proof of \Cref{Affine stratification C(2^n)}.
	This has been done for the~$n = 0$ case in \cite[Section~10.2]{FRChowHa}.
\end{rk}

\begin{proof}[Proof of \Cref{CohaWPL(2^n)=Ppn}]
	By \Cref{Ppn Poincare series} and \Cref{Coha Poincare series} the source and target of the surjective algebra homomorphism~$\Pp_n \to \Coha(\Pbb^1(2^n))$ have the same graded dimension.
	It follows that the map is an isomorphism.
\end{proof}

	

	\addcontentsline{toc}{section}{References}
	\printbibliography

%
%
%
%

\end{document}